\documentclass[reqno]{amsart}
\usepackage[english]{babel}
\usepackage[T1]{fontenc}
\usepackage[utf8]{inputenc}
\usepackage{amsmath}
\usepackage{amsfonts}
\usepackage{array}
\usepackage{eurosym}

\usepackage{amssymb}
\usepackage{mathrsfs}
\usepackage{amsthm}
\usepackage{xfrac}
\usepackage{lmodern}
\usepackage{fix-cm}
\usepackage{listings}
\usepackage{fancyvrb}
\usepackage{enumerate}   
\usepackage{xcolor}
\usepackage{pgf,tikz}
\usepackage{mathrsfs}
\usetikzlibrary{arrows}
\usepackage{float}
\usepackage{subcaption}
\usepackage{hyperref}{}
\usetikzlibrary{patterns}

\usepackage{wasysym}
\usepackage{enumitem}

\theoremstyle{plain}
\newtheorem{theorem}{Theorem}[section]
\newtheorem{lemma}[theorem]{Lemma}
\newtheorem{proposition}[theorem]{Proposition}
\newtheorem{corollary}[theorem]{Corollary}

\newtheorem{conjecture}[theorem]{Conjecture}

\theoremstyle{definition}
\newtheorem{definition}[theorem]{Definition}
\newtheorem{setup}[theorem]{Setup}
\newtheorem{example}[theorem]{Example}
\newtheorem{question}[theorem]{Question}

\newtheorem*{acknowledgement}{Acknowledgement}
\theoremstyle{remark}
\newtheorem{remark}[theorem]{Remark}

\newcommand{\ini}{\textsl{in}_<}
\newcommand{\vv}{\mathbf{v}}

\newcommand{\CC}{\mathcal{C}}
\newcommand{\lk}{\text{lk}_{\Delta_<}}
\usepackage{lscape}
\usepackage{rotating}

\newcommand{\overbar}[1]{\mkern 1.5mu\overline{\mkern-3mu#1\mkern-1.5mu}\mkern 1.5mu}

\newcommand{\rad}{1.5 pt}
\usetikzlibrary{decorations.markings}
\usepackage{subcaption}
\renewcommand{\thesubfigure}{\Alph{subfigure}}

\title{$(S_2)$-condition and Cohen-Macaulay binomial edge ideals}
\author{Alberto Lerda}
\address{University of Trento}
\email{alberto.lerda@studenti.unitn.it}
\author{Carla Mascia}
\address{University of Trento}
\email{carla.mascia@unitn.it}
\author{Giancarlo Rinaldo}
\address{University of Trento}
\email{giancarlo.rinaldo@unitn.it}
\author{Francesco Romeo}
\address{University of Trento}
\email{francesco.romeo-3@unitn.it }

\begin{document}

\maketitle
\begin{abstract}
We describe the simplicial complex $\Delta$ such that the initial ideal of $J_G$ is the Stanley–Reisner ideal of $\Delta$. By $\Delta$ we show that if $J_G$ is $(S_2)$ then $G$ is accessible. We also characterize all accessible blocks with whiskers of cycle rank 3 and we define a new infinite class of accessible blocks with whiskers for any cycle rank. 
Finally, by using a computational approach, we show that the graphs with at most 12 vertices whose binomial edge ideal is Cohen-Macaulay are all and only the accessible ones. 

\end{abstract}
\section*{Introduction}

Binomial edge ideals have been introduced  in \cite{HHHKR} and, independently, in \cite{Oh}. They are associated to finite simple graphs, in fact they arise from the $2$-minors of a $2 \times n$ matrix related to the edges of a graph with $n$ vertices.
The problem of finding a characterization of Cohen–Macaulay binomial edge ideals has been studied intensively by many authors. There are several attempts at this problem available for some families of graphs. Some papers in this direction are \cite{EHH}, \cite{RR}, \cite{R1}, \cite{KS}, \cite{BN}, \cite{BMS}, \cite{R2}, \cite{JK}, \cite{Mo}, \cite{ERT},  and \cite{BMS2}. In the latter, the authors introduce two combinatorial properties strictly related to the Cohen-Macaulayness of binomial edge ideals: accessibility and strongly unmixedness. 
In particular, they prove  
\[
J_G \mbox{ strongly unmixed} \implies J_G \mbox{ Cohen-Macaulay} \implies G \mbox{ accessible}.
\]
In the same article, they show that the three conditions are equivalent for chordal and traceable  graphs.

On the other hand, a fundamental condition to describe Cohen-Macaulay modules is the so-called \emph{Serre's condition} $(S_r)$. N. Terai, in \cite{Te}, translates this condition into nice combinatorial terms for the class of squarefree monomial ideals. In general, for any ideal $I \subseteq S$, it holds true 
\[
S/I \mbox{ Cohen-Macaulay} \implies S/I \mbox{ satisfies Serre's condition } (S_2).
\]

The main aim of this work is to combine all the above-mentioned algebraic and combinatorial notions, showing that 
\[
S/J_G \mbox{ satisfies Serre's condition } (S_2) \implies G \mbox{ accessible},
\]
and finding a large family of graphs that satisfies all of them. To reach the goal, in Section \ref{sec: S2}, we describe the simplicial complex $\Delta_<$ such that $\ini(J_G) = I_{\Delta_<}$, for any term order $<$. It is well known that $\ini(J_G)$ is a squarefree monomial ideal. In \cite{CV}, the authors prove that a binomial edge ideal $J_G$ satisfies the Serre's condition $(S_2)$ if and only if $\ini(J_G)$ satisfies it, as well. We exploit this fact and the knowledge of $\Delta_<$ to prove that if $J_G$ satisfies $(S_2)$-condition, then $G$ is accessible, improving the results of \cite{BMS2}.

In Section \ref{sec: blocks of acc graphs}, we focus on accessible graphs. In particular, in Proposition \ref{pro:block_whiskers} we show that any accessible graph induces, in a natural way, blocks with whiskers that are accessible, too. The latter gives us a sufficient condition for having non-Cohen-Macaulay binomial edge ideals. In literature, many of the examples of non-Cohen-Macaulay $J_G$ are blocks with whiskers (see \cite{R1}, \cite{R2}, \cite{BMS}, and \cite{BMS2}). This fact and Proposition \ref{pro:block_whiskers} motivate us to study accessible blocks with whiskers. In particular, we identify all the blocks with whiskers having cycle rank 3 (See Figure \ref{Fig:ClassesCR3}) and among them we characterize the accessible ones (see Figures \ref{fig:chainsCR3} and \ref{fig:K4CR3}). This represents a further step in the study of graphs with a given cycle rank, following the $3$rd author's works done in \cite{R1} and \cite{R2}, where he classifies the complete intersection ideals by means of cycle rank (0 in that case), and all the Cohen-Macaulay graphs with cycle rank 1 and 2. Moreover, we observe that the number of blocks with whiskers of a given cycle rank is finite  (Lemma \ref{lem:cyclerank} and Lemma \ref{lem:P2_P3}). 
We define a rich family of blocks with whiskers of a given cycle rank that we call \textit{chain of cycles} (see Definition \ref{def:chaincycle}), and we provide necessary conditions for being accessible. Finally, under certain hypotheses on the structure of these graphs (see Setup \ref{setup: chain}), we find an infinite subfamily of chain of cycles $G$ for which all the above-mentioned algebraic and combinatorial properties for $G$ and $J_G$ are equivalent (see Theorem \ref{theo:chaincycle}).

In the last section, we give a computational classification of all the indecomposable Cohen-Macaulay binomial edge ideals of graphs with at most 12 vertices (see Theorem \ref{Theo:ByComputer}). This result has been obtained by using a C++ implementation of the algorithms related to the combinatorial properties of accessibility, $(S_2)$-condition and strongly-unmixedness. The implementation is freely downloadable from the website \cite{LMRR}. This computation and Theorem \ref{theo:S2} lead us to the following.

\begin{conjecture}
Let $G$ be a graph. Then $G$ is accessible if and only if $S/J_G$ satisfies Serre's condition $(S_2)$. 
\end{conjecture}

In \cite{BMS2}, the authors conjecture that accessible graphs are the only with Cohen-Macaulay binomial edge ideal. Our computation supports this conjecture. Finally, among the blocks that, after adding suitable whisker, satisfy Theorem \ref{Theo:ByComputer} we find two polyhedral graphs, hence Question \ref{question} naturally arises.

\section{Preliminaries}
In this section we recall some concepts and notation on graphs, simplicial complexes and binomial edge ideals  that we will use in the article (see also \cite{HHHKR},\cite{RR},\cite{BMS}, \cite{Te}).

Throughout this work, all graphs will be finite and simple, namely undirected graphs with no loops nor multiple edges. Given a graph $G$, we denote by $V(G)$ and $E(G)$ its vertex and edge set, respectively.
Let $G$ be a graph with vertex set $[n] = \{1, \dots, n\}$. If $e = \{u,v\} \in E(G)$, with $u,v \in V(G)$, we say that $u$ and $v$ are \textit{adjacent} and the edge $e$ is \textit{incident} with $u$ and $v$. We denote by $N_G(v)$ (or simply $N(v)$ if $G$ is clear from the context) the set of vertices of $G$ adjacent to $v$. The \textit{degree} of $v \in V(G)$, denoted $\deg v$, is the number of edges incident with $v$. An edge $\{u,v\} \in E(G)$, where $\deg v = 1$, is called \textit{whisker} on $u$. Given $u,v \in V(G)$, a \textit{path} from $v$ to $u$ of length $n$ is a sequence of vertices $v=v_0, \dots, v_n=u \in V(G)$, such that for each $1 \leq i,j \leq n, \{v_{i-1},v_i\} \in E(G)$ and $v_i \neq v_j$ if $i\neq j$.
A subset $C$ of $V(G)$ is called a \textit{clique} of $G$ if for all $u, v \in C$, with $u \neq v$, one has $\{u,v\} \in E(G)$. A \textit{maximal clique} is a clique that cannot be extended by including one more adjacent vertex. A vertex $v$ is called \textit{free vertex} of $G$ if it belongs to only one maximal clique, otherwise it is called an \textit{inner vertex} of $G$.


If $T \subseteq V(G)$, we denote by $G \setminus T$ the induced subgraph of $G$ obtained by removing from $G$ the vertices of $T$ and all the edges incident in them. A set $T \subset V(G)$ is called \textit{cutset} of $G$ if $c_G(T \setminus \{v\}) < c_G(T)$ for each $v \in T$, where $c_G(T)$ (or simply $c(T)$, if the graph is clear from the context) denotes the number of connected components of $G \setminus T$. We denote by $\mathcal{C}(G)$ the set of all cutsets of $G$. When $T \in \mathcal{C}(G)$ consists of one vertex $v$, $v$ is called a \textit{cutpoint}. A connected induced subgraph of $G$ that has no cutpoint and is maximal with respect to this property is called a \textit{block}. 

A subgraph $H$ of $G$ \textit{spans} $G$ if $V(H) = V(G)$. In a connected graph $G$, a \textit{chord} of a tree $H$ that spans $G$ is an edge of $G$ not in $H$. The number of chords of any spanning tree of a connected graph $G$, denoted by $m(G)$, is called the \textit{cycle rank} of $G$ and it is given by $m(G) = |E(G)| - |V (G)| + 1$.

Let $S = \mathbb{K}[\{x_i, y_j\}_{1 \leq i,j \leq n}]$ be the polynomial ring in $2n$ variables with coefficients in a field $\mathbb{K}$.  
Define $f_{ij} = x_iy_j - x_jy_i \in S$.
The \textit{binomial edge ideal} of $G$, denoted by $J_G$, is the ideal generated by all the binomials $f_{ij}$, for $ i<j$ and $\{i,j\} \in E(G)$.

The cutsets of a graph $G$ are essential tools to describe the primary decomposition and several algebraic properties of $J_G$. Let $T \in \mathcal{C}(G)$ and let $G_1, \dots, G_{c(T)}$ denote the connected components of $G \setminus T$. Let
\[
P_T(G) = \left( \bigcup_{i \in T} \{x_i, y_i\}, J_{\tilde{G}_1}, \dots, J_{\tilde{G}_{c(T)}} \right) \subseteq S
\]
where $\tilde{G}_i$, for $i=1, \dots, c(T)$, denotes the complete graph on $V(G_i)$. It holds 
\begin{equation}\label{Eq:primarydec}
J_G = \bigcap_{T \in \mathcal{C}(G)} P_T(G). 
\end{equation}

A graph $G$ is {\em decomposable},   if there exist two subgraphs $G_1$ and $G_2$ of $G$,  and a decomposition
  $G=G_1\cup G_2$
  with $\{v\}=V(G_1)\cap V(G_2)$, where   $v$ is a free vertex of $G_1$ and $G_2$. If $G$ is  not decomposable, we call it {\em indecomposable}.

Let $H$ be a graph. The \textit{cone} $G$ of $v$ on $H$ is the graph
with $V(G)=V(H) \cup \{v\}$ and edges $E(G)= E(H) \cup \{\{v,w\} \ | \ w \in V (G)\}$.

A cutset $T$ of $G$ is said \textit{accessible} if there exists $t \in T$ such that $T \setminus \{t\} \in \mathcal{C}(G)$. $G$ is said \textit{accessible} if $J_G$ is unmixed and $\mathcal{C}(G)$ is an accessible set system, that is all non-empty cutsets of $G$ are accessible. 

To describe the reduced Gröbner basis of $J_G$, in \cite{HHHKR} the following concept has been introduced.  Let $i$ and $j$ be two vertices of $G$ with $i < j$. A path $i=i_0, i_1, \dots, i_r = j$ from $i$ to $j$ is called \textit{admissible} if
\begin{enumerate}[label=(\roman*)]
\item $i_k \neq i_\ell$ for $k \neq \ell$;
\item for each $k=1, \dots, r-1$ one has $i_k < i$ or $i_k > j$;
\item for any $\{j_1, \dots, j_s\} \ subset \{i_1, \dots, i_r\}$, the sequence $i, j_1, \dots, j_s, j$ is not a path.
\end{enumerate}

Given an admissible path $\pi : i=i_0, i_1, \dots, i_r = j$ from $i$ to $j$, where $i <j$, define the monomial 
\[
u_{\pi} = \left(\prod_{i_k > j} x_{i_k} \right) \left(\prod_{i_\ell < i\phantom{j}} y_{i_\ell} \right).
\]

\begin{theorem}\label{theo: groebner basis}
Let $G$ be a graph on $[n]$. Let $<$ be the lexicographic order on $S$ induced by $x_1 > x_2 > \cdots > x_n > y_1 > \cdots > y_n$. Then the set 
\[
\mathcal{G} = \bigcup_{i<j} \{u_\pi f_{ij} \ | \ \pi \text{ is an admissible path from } i \text{ to } j\}
\]
is the reduced Gröbner basis of $J_G$ with respect to $<$.
\end{theorem}

A finitely generated graded module $M$ over a Noetherian graded $\mathbb{K}$-algebra
$R$ is said to satisfy the \textit{Serre's condition} $(S_r)$, or simply $M$ is an $(S_r)$ \textit{module} if, for all $\mathfrak{p} \in \mathrm{Spec}(R)$, the inequality 
\[
\mathrm{depth}\  M_{\mathfrak{p}} \geq \min (r, \dim M_{\mathfrak{p}})
\]
holds true. The Serre's conditions are strictly connected to the Cohen-Macaulayness of a module, in fact $M$ is Cohen–Macaulay if and only if it is an $(S_r)$ module for all $r \geq 1$.

A \textit{simplicial complex} $\Delta$ on the set of vertices $[n]$ is a collection of subsets of $[n]$ which is closed under taking subsets, that is, if $F \in \Delta$ and $F' \subseteq F$, then also $F' \in \Delta$. Every element $F \in \Delta$ is called a \textit{face} of $\Delta$; the \textit{size} of a face $F$ is defined to be $|F|$, that is, the number of elements of $F$, and its dimension is defined to be $|F| -1$. The dimension of $\Delta$, which is denoted by $\dim (\Delta)$, is defined to be $d-1$, where $d = \max\{|F| \ | \ F \in \Delta\}$. A \textit{facet} of $\Delta$ is a maximal face of $\Delta$ with respect to inclusion. Let $\mathcal{F}(\Delta)$ denote the set of facets of $\Delta$. It is clear that $\mathcal{F}(\Delta)$ determines $\Delta$. A set $N \subseteq [n]$ that does not belong to $\Delta$ is called \textit{nonface} of $\Delta$. We say that $\Delta$ is \textit{pure} if all facets of $\Delta$ have the same size. The link of $\Delta$ with respect to a face $F \in \Delta$, denoted by $\text{lk}_{\Delta}(F)$, is the simplicial complex 
\[
\text{lk}_{\Delta}(F) = \{ G \subseteq [n] \setminus F \ | \ G \cup F \in \Delta\}.
\]
A simplicial complex $\Delta$ is called \textit{connected} if, for every $F,G \in \mathcal{F}(\Delta)$, there exists a sequence of facets $F=F_0, \dots, F_m=G$ such that, for every $0 \leq i,j \leq m-1$, we have $F_i \cap F_{i+1} \neq \emptyset$ and $F_i \neq F_j$, where $i \neq j$. We say that the sequence $F=F_0, \dots, F_m=G$ connects $F$ and $G$.

Let $R = \mathbb{K}[z_1,\dots, z_k]$ be the polynomial ring in $k$ variables over a field $\mathbb{K}$, and let $\Delta$ be a simplicial complex on $[k]$. For every subset $F\subseteq[k]$, we set $z_F = \prod_{i \in F} z_i$. The \textit{Stanley–Reisner ideal of $\Delta$} over $\mathbb{K}$ is the ideal $I$ of $R$ which is generated by those squarefree monomials $z_F$ with $F \not \in \Delta$. In other words, $I_{\Delta}= (z_F \ | \ F \in \mathcal{N}(\Delta))$, where $\mathcal{N}(\Delta)$ denotes the set of minimal nonfaces of $\Delta$ with respect to inclusion. The \textit{Stanley–Reisner ring of $\Delta$} over $\mathbb{K}$, denoted by $\mathbb{K}[\Delta]$, is defined to be $\mathbb{K}[\Delta] = R/I_{\Delta}$.

A simplicial complex $\Delta$ is said to satisfy \textit{Serre’s condition} $(S_r)$\textit{ over} $\mathbb{K}$, or simply $\Delta$ is an $(S_r)$ simplicial complex over $\mathbb{K}$, if the Stanley–Reisner ring $\mathbb{K}[\Delta]$ of $\Delta$ satisfies
Serre’s condition $(S_r)$. An immediate consequence of \cite[Theorem 1.4]{Te} is the following result that provides a useful combinatorial tool to check if $\Delta$ is $(S_2)$. 
\begin{proposition}\label{prop: link S2}
Let $\mathbb{K}$ be a field and $\Delta$ a simplicial complex. Then $\Delta$ is $(S_2)$ over $\mathbb{K}$ if and only if, for every face $F \in \Delta$ with $\dim(\mathrm{lk}_{\Delta}(F)) \geq 1$, the simplicial complex $\mathrm{lk}_{\Delta}(F)$ is connected. In particular, the $(S_2)$ property of a simplicial complex is independent from the base field. \\
\end{proposition}
\section{Simplicial complex of binomial edge ideals and ($S_2$)-condition}\label{sec: S2}
The aim of this section is to prove that if $S/J_G$ satisfies the Serre's condition $(S_2)$, then $G$ is an accessible graph. 

Let $<$ be a monomial order on $S$ and $\ini(I)$ denote the initial ideal of an ideal $I$ with respect to $<$. A consequence of \cite[Theorem 1.3]{CV} is that, if $I$ is an ideal and $\ini(I)$ is a square-free monomial ideal, then, for any $r \in \mathbb{N}$, $S/I$ satisfies Serre’s condition $(S_r)$ if and only if $S/ \ini(I)$ does. Since $\ini(J_G)$ is square-free (see \cite[Section 3.2]{CV}), it follows that to study the $(S_2)$ condition for $S/J_G$ it is sufficient to study it for $S/\ini(J_G)$. 

%


From now on, we fix the lexicographic order on $S$ induced by $x_1 > x_2 > \cdots > x_n > y_1 > \cdots > y_n$. 

Let $T \in \mathcal{C}(G)$ and let $G_1, \dotsm, G_{c(T)}$ be the connected components induced by $T$. By Theorem \ref{theo: groebner basis}, it follows immediately
\[
\ini(J_G) = \left(x_iy_ju_\pi \ |  \ \pi \text{ is an admissible path from } i \text{ to } j, \text{ with } i < j\right),
\]
and
\[
\ini (P_T(G)) = \left(\bigcup_{t \in T} \{x_t, y_t\}\right) + \sum_{k=1}^{c(T)} \left(x_iy_j \ | \ i, j \in V(G_k) \text{ and } i <j \right).
\]

Moreover, thanks to \cite{CDG}, it holds
\begin{equation}\label{eq: in J_G}
\ini(J_G)=\bigcap_{T\in \mathcal{C}(G)} \ini(P_T(G)).
\end{equation}

Define
\[
P_T(\vv)=\left(\bigcup_{t\in T} \{x_t,y_t\}\right)+\sum_{k=1}^{c(T)} \left( \{x_i \mid i\in V(G_k), i<v_k\}\cup \{y_j \mid j\in V(G_k), j>v_k\}\right)
\]
where $\vv=(v_1,\ldots,v_{c(T)})\in V(G_1)\times \cdots \times V(G_{c(T)})$.

\begin{lemma}\label{lemma: in P_T}
Let $G$ be a graph. Let $T \in \mathcal{C}(G)$ and let $G_1, \dotsm, G_{c(T)}$ be the connected components induced by $T$. Then 
\[
\ini (P_T(G)) = \bigcap_{\vv\in V(G_1)\times \cdots \times V(G_{c(T)})} P_T(\vv).
\]
\end{lemma}

\begin{proof}
$`` \subseteq"$ Let $u$ be a generator of $\ini (P_T(G))$. If $u \in \{x_t, y_t\}$ for $t \in T$, then $u \in P_T(\vv)$, for all $\vv\in V(G_1)\times \cdots \times V(G_{c(T)})$. Let $u = x_iy_j$, with $i<j$ and $i,j \in V(G_k)$, for some $k=1, \dots, c(T)$, and consider $v_k$, the $k$-th component of $\vv$. When $v_k \leq i$, then $y_j \in P_T(\vv)$, when $v_k >i$, then $x_i \in P_T(\vv)$. Hence, the monomial $x_iy_j \in P_T(\vv)$ for all $\vv\in V(G_1)\times \cdots \times V(G_{c(T)})$.

$`` \supseteq"$ Let $u$ be a generator of $\bigcap_{\vv\in V(G_1)\times \cdots \times V(G_{c(T)})} P_T(\vv)$. If $x_t$ divides $u$, for some $t\in T$, then $u \in \ini (P_T(G))$, as well. Assume that $x_t$ does not divide $u$, for any $t\in T$.
For $k=1, \dots, c(T)$, denote $J_k =\left(x_iy_j \ | \ i,j \in V(G_k) \text{ and } i<j\right)$ and $I_{v_k} = \left( \{x_i \mid i\in V(G_k), i<v_k\}\cup \{y_j \mid j\in V(G_k), j>v_k\}\right)$, for $v_k \in V(G_k)$. Then 
\[
\ini (P_T(G)) = \left(\bigcup_{t\in T} \{x_t,y_t\}\right)+\sum_{k=1}^{c(T)} J_k
\]
and 
\[
P_T(\vv) = \left(\bigcup_{t\in T} \{x_t,y_t\}\right)+\sum_{k=1}^{c(T)} I_{v_k}.
\]
Note that $I_{v_k}$ and $J_{k}$ are both ideals of $S_k = \mathbb{K}[x_i, y_i]_{i \in V(G_k)}$. Moreover, $I_{v_k}$ and $I_{v_h}$, with $v_k \in G_k$, $v_h \in G_h$ and $k \neq h$, are defined on a disjoint set of variable, and the same holds for the $J_k$'s. It is sufficient to prove that 
\[
J_k \supseteq \bigcap_{v_k \in V(G_k)} I_{v_k}.
\]

Assume that $u \in \bigcap_{v_k \in V(G_k)} I_{v_k}$. Note that $u$ can not be the product of only $x_i$'s (resp. $y_j$'s). Indeed, when $v_k = \min \{a \  | \ a \in V(G_k)\}$ (resp. $v_k = \max \{b \  | \ b \in V(G_k)\}$), then no $x_i$ belongs to $I_{v_k}$ (resp. no $y_j$ belongs to $I_{v_k}$). Now, suppose, by contradiction, that for any $x_iy_j$ that divides $u$, it holds $i >j$. Set $v_k = \min \{i \ | \ x_i \text{ divides } u \}$. Then all the $x_i$'s and $y_j$'s that divide $u$ do not belong to $I_{v_k}$, namely $u \not \in I_{v_k}$. It follows that if $x_iy_j$ divides $u$, then $i<j$ and  $u \in J_k$. 
\end{proof}

Let $T \in \mathcal{C}(G)$ and let $G_1, \dots, G_{c(T)}$ denote the connected components of $G \setminus T$. For $i=1, \dots, c(T)$, let $|V(G_i)| = m_i$ and $V(G_i) = \{v_1^i, \dots, v_{m_i}^i\}$. Given $\vv = \left(v_{j_1}^1, \dots, v_{j_{c(T)}}^{c(T)}\right) \in V(G_1) \times \cdots \times V(G_{c(T)})$, define 
\[
F(T, \vv) = \bigcup_{i= 1}^{c(T)} \left\lbrace\{ y_j \ | \ j \leq v_{j_i}^i\} \cup \{ x_j \ | \ j \geq v_{j_i}^i\}\right\rbrace.
\]

Since $\ini(J_G)$ is a squarefree monomial ideal, then there exists a unique simplicial complex $\Delta_{<}$ such that $\ini (J_G) = I_{\Delta_{<}}$. By Equation (\ref{eq: in J_G}) and Lemma \ref{lemma: in P_T}, we obtain the following description of $\Delta_{<}$.

\begin{corollary}\label{cor: delta}
Let $G$ be a graph. Then $\ini (J_G) = I_{\Delta_{<}}$, where 
\[
\mathcal{F}(\Delta_{<})=\bigcup_{T\in \mathcal{C}(G)} \{ F(T, \vv): \vv\in V(G_1)\times \cdots \times V(G_{c(T)}) \} .
\]
\end{corollary}

For a graded $S$-module $M$ we denote by $H(t) = \sum_{i=0}^d (h_i(M)) t^i/(1-t)^d$ the Hilbert series of $M$ and by $h=(h_0,\ldots,h_d)$ its $h$-vector. The following result, by a well known formula that relates $f$-vector with $h$-vector, gives a way to compute the invariant by $\Delta_{<}$ as defined above.

\begin{corollary}\label{cor:h vect}
The $h$-vector of $\Delta_{<}$ is
\[
 h_k=\sum_{i=0}^{k}(-1)^{k-i}\binom{d-i}{k-i} f_{i-1}(\Delta_{<}).
\]
for $k=0,\ldots,d$.
\end{corollary}

In \cite{BN}, authors provide a formula to compute the multiplicity of $S/J_G$. By knowing $\Delta_{<}$ such that $\ini(J_G) = I_{\Delta_{<}}$ and by Corollary \ref{cor:h vect}, one can easily obtain another simple way to get the multiplicity. \\

%

In the following, we deeply use the simplicial complex $\Delta_{<}$ defined in Corollary \ref{cor: delta} to prove that if $S/J_G$ satisfies the Serre's condition $(S_2)$, then the graph $G$ is accessible. Nevertheless, we observe that the simplicial complex is strongly related to the chosen monomial order also for very simple graphs, as the following Example shows.
\begin{example}
Let $G=P_2$ be the path on 3 vertices with $E(G) = \{\{1,2\}, \{2, 3\}\}$ and fix the lexicographic order on $S$ induced by $x_1 > x_2 > x_3 > y_1 > y_2 > y_3$. Then, $\mathcal{C}(G) =\{ \emptyset, \{2\}\}$ and $I_{\Delta_{<}} = (x_1y_2, x_2y_3)$, where
\begin{align*}
\Delta_{<} &= \{F(\emptyset, (1)), F(\emptyset, (2)), F(\emptyset, (3)), F(\{2\}, (1,3))\} \\
&= \{ \{x_1,y_1,x_2,x_3\}, \{y_1,x_2,y_2,x_3\}, \{y_1,y_2,x_3,y_3\}, \{x_1,y_1,x_3,y_3\}\}.
\end{align*}
One can immediately observe that all the facets in $\Delta_{<}$ contain the variables $y_1$ and $x_3$. 
Consider now the same graph but with a different vertex labelling with $E(G) = \{\{1,3\}, \{2, 3\}\}$. Fix the same term order for $S$. Then, $\mathcal{C}(G) =\{ \emptyset, \{3\}\}$ and  $I_{\Delta_{<}} = (x_1y_3, x_2y_3, x_1y_2x_3)$, where
\begin{align*}
\Delta_{<} &= \{F(\emptyset, (1)), F(\emptyset, (2)), F(\emptyset, (3)), F(\{3\}, (1,2))\} \\
&= \{ \{x_1,y_1,x_2,x_3\}, \{y_1,x_2,y_2,x_3\}, \{y_1,y_2,x_3,y_3\}, \{x_1,y_1,x_2,y_2\}\}.
\end{align*}
In this case, only the variable $y_1$ is contained in all the facets of $\Delta_{<}$. This implies that the two simplicial complexes are not isomorphic.
\end{example}

\begin{remark}\label{remark: union cutset}
Let $G$ be a graph on $[n]$. Let $T \in \mathcal{C}(G)$ and $v \in T$ be a cutpoint of $G$ that induces two connected components, $H_1$ and $H_2$. For $i=1,2$, let $T_i \subseteq T \cap V(H_i)$. If $T_1$ and $T_2$ are cutsets of $G$, then $T_1 \cup T_2$ is a cutset of $G$. 
\end{remark}

\begin{lemma}\label{lemma: T accessible}
Let $G$ be a graph on $[n]$. Let $T \in \mathcal{C}(G)$ and $v \in T$ be a cutpoint of $G$ that induces two connected components, $H_1$ and $H_2$. For $i=1,2$, let $T_i = T \cap V(H_i)$. If $S_1 = T_1 \cup \{v\}$ and $S_2=T_2 \cup \{v\}$ are accessible cutsets of $G$, then $T$ is an accessible cutset of $G$. 
\end{lemma}

\begin{proof}
By hypothesis, $S_1$ and $S_2$ are accessible, that is there exist $v_1 \in S_1$ and $v_2 \in S_2$ such that $S_1 \setminus \{v_1\}, S_2 \setminus \{v_2\} \in \mathcal{C}(G)$. If $v_1 = v_2 = v$, then, by Remark \ref{remark: union cutset}, $T_1 \cup T_2 = T \setminus \{v\}$ is a cutset of $G$, namely $T$ is accessible. If at least one between $v_1$ and $v_2$ is not $v$, assume $v_1 \neq v$, then, by Remark \ref{remark: union cutset}, $S_1 \setminus \{v_1\} \cup T_2 = T \setminus \{v_1\}$ is a cutset of $G$, namely $T$ is accessible.
\end{proof}

\begin{remark}\label{remark: T accessible cutpoint}
Let $G$ be a graph and $T \in \CC(G)$. If all the cutset $T'$, with $T' \subset T$, are accessible, then $T$ contains a cutpoint. The proof of this fact is the same of \cite[Lemma 4.1]{BMS2}. 
\end{remark}

\begin{theorem}\label{theo:S2}
Let $G$ be a graph such that $S/J_G$ satisfies the Serre’s condition $(S_2)$. Then $G$ is an accessible graph.
\end{theorem}

\begin{proof}
To prove the statement, we suppose that $G$ is not accessible and we show that $S/J_G$ does not satisfy the Serre’s condition $(S_2)$. If $G$ is not accessible then $J_G$ is not unmixed or $\mathcal{C}(G)$ is not an accessible set system. If $J_G$ is not unmixed, then it is known that the $(S_2)$-condition is not satisfied. Hence, we can suppose that $J_G$ is unmixed but $\mathcal{C}(G)$ is not an accessible set system. Let $T \in \mathcal{C}(G)$ be the non-empty cutset with the minimum cardinality such that $T \setminus \{v\} \not \in \mathcal{C}(G)$, for every $v \in T$. Let $T = \{w_1, \dots, w_k\}$, with $k >1$, and $G_1, \dots, G_{k+1}$ be the connected components of $G \setminus T$. For $i=1, \dots, k+1$, let $|V(G_i)| = m_i$ and $V(G_i) = \{v_1^i, \dots, v_{m_i}^i\}$.

Fix the lexicographic order on $S$ induced by the total order 
\begin{equation}
w_1 < \cdots < w_k < v_1^1 < \dots <v_{m_1}^1 < \cdots < v_1^{k+1} < \dots  < v_{m_{k+1}}^{k+1} \label{eq: order} \tag{$\star$}
\end{equation}

Thanks to \cite[Theorem 1.3]{CV}, it is sufficient to prove that $S/\ini(J_G)$ does not satisfies the Serre’s condition $(S_2)$. 

Consider $\vv = (v_{m_1}^1 , \dots, v_{m_{k+1}}^{k+1}) \in V(G_1) \times \cdots \times V(G_{k+1})$ and
\[
F(T, \vv) = \bigcup_{i =1}^{k+1} \left\lbrace y_{v_1^i}, \dots, y_{v_{m_i}^i}, x_{v_{m_i}^i}\right\rbrace \in \mathcal{F}(\Delta_{<}).
\]
The set
\[
F= \bigcup_{i =1}^{k} \left\lbrace y_{v_1^i}, \dots, y_{v_{m_i}^i} \right\rbrace \cup \left\lbrace y_{v_1^{k+1}}, \dots, y_{v_{m_{k+1}}^{k+1}}, x_{v_{m_{k+1}}^{k+1}}\right\rbrace
\]
is a subset of $F(T, \vv)$, that is a face of $\Delta_{<}$. Consider the link of  $\Delta_{<}$ with respect to $F$. The sets $A = \{x_{v_{m_1}^1}, \dots, x_{v_{m_k}^k}\}$ and $B = \{y_{w_1}, \dots, y_{w_k}\}$ belong to $\text{lk}_{\Delta_{<}}(F)$. In fact, thanks to the order $(\star)$, $A \cap F = \emptyset$ and $A \cup F = F(T, \vv) \in \mathcal{F}(\Delta_{<})$, whereas, $B \cap F = \emptyset$ and $B \cup F = F(\emptyset, \mathbf{u}) \in \mathcal{F}(\Delta_{<})$, where $\mathbf{u} = (v_{m_{k+1}}^{k+1})$. Since $|A| = |B| = k >1$, it follows $\dim \mathrm{lk}_{\Delta_{<}}(F) \geq 1$. Assume, by contradiction, that $\text{lk}_{\Delta_{<}}(F)$ is connected, that is there exists a sequence of facets $A = F_0, F_1, \dots, F_{t+1} = B$ of $\text{lk}_{\Delta_{<}}(F)$ such that, for every $0 \leq i,j \leq t$, $F_i \cap F_{i+1} \neq \emptyset$ and $F_i \neq F_j$ when $i \neq j$. First of all, suppose that $F_{t} \cap B = \{y_{w_i}\}$, for some $i=1, \dots, k$. 
Without loss of generality, assume $i=1$. Then there exists $F(T',\mathbf{\overbar{v}}) \in \mathcal{F}(\Delta_{<})$ such that $F(T',\mathbf{\overbar{v}}) = F_t \cup F$. 
Note that $y_{w_1} \in F(T',\mathbf{\overbar{v}})$ but $y_{w_i} \not \in F(T',\mathbf{\overbar{v}})$, for $i \neq 1$, otherwise $F_t \cap B \supset \{y_{w_1}\}$. 
Since  $y_{w_i} \not \in F(T',\mathbf{\overbar{v}})$, for $i \neq 1$, and $y_{v} \in F(T',\mathbf{\overbar{v}})$, for $v \in (V(G) \setminus T) \cup \{w_1\}$, that is either $v = w_1$ or $v > w_k$, then $x_{w_i} \not \in F(T',\mathbf{\overbar{v}})$, for $i \neq 1$. 
From the fact that $x_{w_i}, y_{w_i} \not \in F(T',\mathbf{\overbar{v}})$, it follows that $T' = \{w_2, \dots, w_k\}$ and $\mathbf{\overbar{v}} = (v_{m_2}^2, \dots, v_{m_{k+1}}^{k+1})$. This implies that $T' = T \setminus \{w_1\} \in \mathcal{C}(G)$, but this is in contradiction with the hypothesis that $T$ is not an accessible cutset. 

Now, suppose that $|F_{t} \cap B| > 1$. Note that $|F_{t} \cap B| <k$, otherwise $F_{t} \cap B= B$, that is $F_t = F_{t+1} =B$, which contradicts the hypothesis on $F_i$. Without loss of generality, assume $F_{t} \cap B= \{y_{w_1}, \dots, y_{w_a}\}$, with $1<a<k$. There exists $F(T'',\mathbf{\overbar{v}'}) \in \mathcal{F}(\Delta_{<})$ such that $F(T'',\mathbf{\overbar{v}'}) = F_t \cup F$. For $i >a$, it holds $y_{w_i} \not \in F_t$, hence $y_{w_i} \not \in F(T'',\mathbf{\overbar{v}'})$. 
Since $y_{v} \in F(T',\mathbf{\overbar{v}})$, for every $v  > w_k$, then $x_{w_i} \not \in F(T',\mathbf{\overbar{v}})$, for $i=1, \dots, k$. Therefore, $x_{w_i}, y_{w_i} \not \in F(T'',\mathbf{\overbar{v}'})$ for $i >a$ and $T'' = \{w_{a+1}, \dots, w_k\}$. By hypothesis, $T$ is the smallest not accessible cutset, then any cutset which is a proper subset of $T$ is accessible. Since $T'' \subset T$, then $T''$ is accessible and, by Remark \ref{remark: T accessible cutpoint}, $T''$ contains a cutpoint, we say $w_{a+1}$. Then $w_{a+1}$ induces two connected components, $H_1$ and $H_2$. Let $T_i = T \cap V(H_i)$, for $i=1,2$. For $i=1,2$, $T_i \cup \{w_{a+1}\}$ is a cutset of $G$. By the minimality of $T$, both $T_1 \cup \{w_{a+1}\}$ and $T_2 \cup \{w_{a+1}\}$ are accessible cutsets of $G$. By Lemma \ref{lemma: T accessible}, also $T = T_1 \cup T_2 \cup \{w_{a+1}\}$ is an accessible cutset, which is a contradiction.

It follows that $\text{lk}_{\Delta_{<}}(F)$ is not connected, and then $S/\ini(J_G)$ does not satisfy the Serre's condition $(S_2)$.

\end{proof}

Let $G$ be a graph such that $J_G$ is unmixed. The following results state that to verify the Serre's condition $(S_2)$ for $S/J_G$ is not necessary to check the link of all the faces $F$ of $\Delta_{<}$.

\begin{proposition}
Let $G$ be a graph on $[n]$, with $n \leq 12$, such that $J_G$ is unmixed. For all monomial order $<$ and all $F \in \Delta_<$ such that $\dim F < \lfloor \frac{n+1}{2} \rfloor $ it holds that $\mathrm{lk}_{\Delta_<}(F)$ is connected. 
\end{proposition}

\begin{proof}
We have implemented a computer program, see \cite{LMRR}, that checks the Serre's condition $(S_2)$ for $S/J_G$. By means of it, we have verified that the statement holds. In particular, there exists a unique family of graphs such that $\mathrm{lk}_{\Delta_<}(F)$ is disconnected for $F \in \Delta_<$ with $\dim F = \lfloor \frac{n+1}{2} \rfloor$, that is the one in Example \ref{ex: star graphs}.
\end{proof}

\begin{example}\label{ex: star graphs}
Let $G$ be a graph on $[n]$ obtained by joining $s+1$ complete graphs $G_1, \dots, G_{s+1}$ such that $G_1 = \cdots = G_s = K_{s+1}$, if $n$ is odd $G_{s+1} = K_{s+1}$, otherwise $G_{s+1} = K_{s+2}$, and $G_i \cap G_j = H$, where $H = K_s$, for all $1 \leq i < j \leq s+1$. See Figure \ref{fig:NotAcc} for an example, with $n=7$. We observe that $\mathcal{C}(G) = \{\emptyset, T\}$, where $T = V(H)$. Moreover, $J_G$ is unmixed but $G$ is a block that is not a complete graph, hence $J_G$ is not Cohen-Macaulay by \cite{BN}. Fix the lexicographic order on $S$ induced by $x_1 > \cdots > x_n > y_1 > \cdots > y_n$. Let $V(H) = \{n-s+1, \dots, n\}$, and consider $F = \{y_1, \dots, y_{n-s}, x_{n-s}\} \in \Delta_<$. Note that $\dim F = n-s = \lfloor \frac{n+1}{2} \rfloor$. The facets of the link of $F$ in $\Delta_<$ are only two: $F(\emptyset, n-s)\setminus F$ and $F(T,(1, \dots, n-s)) \setminus F $, which are respectively $ \{x_{n-s+1}, \dots, x_n\}$ and $\{x_1, \dots, x_{n-s-1}\}$ and they are obviously disjoint. It follows that $\mathrm{lk}_{\Delta_<}(F)$ is disconnected.

\begin{figure}[H]
\centering
 \resizebox{!}{0.35\textwidth}{
\begin{tikzpicture}
\draw   (-4,5)-- (-4,3);
\draw   (-4,3)-- (-5,1);
\draw   (-5,1)-- (-3,1);
\draw   (-4,3)-- (-3,1);
\draw   (-4,5)-- (-5,1);
\draw   (-4,5)-- (-3,1);
\draw   (-4,-1)-- (-5,1);
\draw   (-4,-1)-- (-3,1);
\draw   (-4,-1)-- (-4,3);
\draw   (-1,2)-- (-3,1);
\draw   (-4,3)-- (-1,2);
\draw   (-5,1)-- (-1,2);
\draw   (-7,2)-- (-4,3);
\draw   (-7,2)-- (-5,1);
\draw   (-7,2)-- (-3,1);
\draw [fill=black] (-4,5) circle (2.5pt) node [anchor=south]{$1$};
\draw [fill=black] (-4,3) circle (2.5pt) node [anchor=south west]{$5$};
\draw [fill=black] (-5,1) circle (2.5pt) node [anchor=north east]{$7$};
\draw [fill=black] (-3,1) circle (2.5pt) node [anchor=north west]{$6$};
\draw [fill=black] (-4,-1) circle (2.5pt) node [anchor=north]{$3$};
\draw [fill=black] (-1,2) circle (2.5pt) node [anchor=west]{$2$};
\draw [fill=black] (-7,2) circle (2.5pt) node [anchor=east]{$4$};

\end{tikzpicture}}
\caption{}\label{fig:NotAcc}
\end{figure}
\end{example}



\begin{proposition}\label{lemma: link conn}
Let $G$ be a graph on $[n]$ such that $J_G$ is unmixed. Let $F = \{x_{i_1}, \dots, x_{i_t}, y_{j_1}, \dots, y_{j_s}\} \in \Delta_{<}$, with $1 \leq j_1 < \cdots < j_s < i_1 < \cdots < i_t \leq n$ and $\dim F \leq n -2$. Then $\mathrm{lk}_{\Delta_<}(F)$ is connected. 
\end{proposition}

\begin{proof}
If $F = \emptyset$, then $\mathrm{lk}_{\Delta_<} (F) = \Delta_<$ is connected. In fact, any facets of $\Delta_<$ have a non-empty intersection with a facet $F(\emptyset, v)$, for some $v \in V(G)$, and $F(\emptyset, v_1) \cap F(\emptyset, v_2) \neq \emptyset$, for all $v_1, v_2 \in V(G)$.  Hence, assume $F = \{x_{i_1}, \dots, x_{i_t}, y_{j_1}, \dots, y_{j_s}\} \in \Delta_{<}$, with $1 \leq j_1 < \cdots < j_s < i_1 < \cdots < i_t\leq n$ and $\dim F \leq n -2$. Let $F_1, F_2$ be facets of $\lk(F)$. If $F_1 \cap F_2 \neq \emptyset$, then they are connected and there is nothing to prove. Therefore, we may assume that $F_1 \cap F_2 = \emptyset$. $F \cup F_1$ and $F \cup F_2$ are facets of $\Delta_{<}$ and both of them contain $y_{j_s}$. By Corollary \ref{cor: delta}, there exist $x_a \in F\cup F_1$ and $x_b \in F \cup F_2$ such that $j_s \leq a,b \leq i_1$. Let $a = \min \{a \ | \ x_a \in F \cup F_1 \text{ and } j_s \leq a \leq i_1\}$ and $b = \min \{b \ | \ x_b \in F \cup F_2 \text{ and } j_s \leq b \leq i_1\}$.

Note that, if $a=b=i_1$, then $y_{i_1} \in F\cup F_i$, for $i=1,2$, but $y_{i_1} \not \in F$, then $y_{i_1} \in F_1 \cap F_2$, which is a contradiction since $F_1$ and $F_2$ are supposed to be disjoint. Moreover, if $a, b < i_1$ and $a=b$, then $x_a \in F_1 \cap F_2$, which is a contradiction, as well. Therefore, let $a\neq b$, and, without loss of generality, suppose $a<b$. Consider the facets $F(\emptyset,\vv)$, for $a \leq \vv \leq b$, namely $F(\emptyset,\vv)= \{x_i \ | \ \vv \leq i \leq n\} \cup \{y_j \ | \ 1 \leq j \leq \vv\}$. Note that, for all $a \leq \vv \leq b$, $F(\emptyset, \vv) \cap F = F$, hence $\overbar{F}_{\vv} = F(\emptyset, \vv) \setminus F = \{x_i \ | \ \vv \leq i \leq n, \ i\neq i_1, \dots, i_t\} \cup \{y_j \ | \ 1 \leq j \leq \vv, \ j \neq j_1, \dots, j_s\}$ is a facet of $\lk(F)$. Consider the sequence $F_1, \overbar{F}_a, \overbar{F}_{a+1},\dots, \overbar{F}_b, F_2$ of facets of $\lk(F)$. Note that $F_1 \cap \overbar{F}_a \supseteq \{x_a\}$ and $\overbar{F}_b \cap F_2 \supseteq \{y_b\}$. If $i_1 = j_s +1$, then $a =j_s$ and $b=i_1$, since $\dim F \leq \dim \Delta_{<} -2$, there exists either $i^* > i_1$ such that $x_{i^*} \not \in F$ or $j^* < j_s$ such that $y_{j^*} \not \in F$. It follows that either $\overbar{F}_a \cap \overbar{F}_b \supseteq \{x_{i^*}\}$ or $\overbar{F}_a \cap \overbar{F}_b \supseteq \{y_{j^*}\}$, that is $F_1, \overbar{F}_a, \overbar{F}_b, F_2$ is a sequence of facets of $\lk(F)$ that connects $F_1$ and $F_2$. If $i_1 \neq j_s +1$ and $a+1 \neq i_1$, it holds $\overbar{F}_a \cap \overbar{F}_{a+1} \supseteq \{x_{a+1}\}$ and $\overbar{F}_i \cap \overbar{F}_{i+1} \supseteq \{y_i\}$ for all $i = a+1, \dots, b-1$. If $i_1 \neq j_s +1$ and $a+1 = b =  i_1$, then $\overbar{F}_a \cap \overbar{F}_b = \{y_{i_1}\}$. Hence, $F_1, \overbar{F}_a, \overbar{F}_{a+1},\dots, \overbar{F}_b, F_2$ is a sequence of facets of $\lk(F)$ that connects $F_1$ and $F_2$.
Therefore, $\lk(F)$ is connected.
\end{proof}

\section{Accessible blocks with whiskers}\label{sec: blocks of acc graphs}
In this section we study a particular class of accessible graphs. We know from \cite[Theorem 4.12]{BMS2} and \cite{BN} that if an accessible graph is a block, then it is a complete graph. It arises a natural question:
\begin{center}
``Under which hypotheses a block with whiskers is accessible?'' 
 \end{center}


Let $G$ be a connected graph such that $J_G$ is unmixed and $B$ be a block of $G$. Denote by $W=\{w_1,\dots,w_r\}$ the set of cutpoints of $G$ which are vertices of $B$. Then  
\begin{equation}\label{eq:blockandGi}
 G=B \cup \left(\bigcup_{i=1}^r G_i\right)
\end{equation}
where $V(G_i)\cap V(B)=\{w_i\}$ for $i=1,\ldots,r$, and $B\setminus W, G_1\setminus\{w_1\}, \dots, G_r\setminus\{w_r\}$ are the connected components of $G \setminus W$.


By the decomposition (\ref{eq:blockandGi}), we define a \textit{block with whiskers}, namely $\overbar{B}$, a graph obtained, roughly speaking, by replacing each subgraph  $G_i$  with a whisker. That is 
\begin{enumerate}
    \item $V(\overbar{B})=V(B)\cup \{f_1,\ldots,f_r\}$;
    \item $E(\overbar{B})=E(B)\cup \{\{w_i,f_i\} \ | \ i=1,\ldots, r\}$.
\end{enumerate}
Note that $V(\overbar{B}) = V(G)/\sim$, where the relation $\sim$ identifies each vertex of $B$ with itself and, for $i=1,\dots,r$, if $a,b\in V(G_i)\setminus \{w_i\}$, then $a\sim b$, and we denote by $f_i$ the  equivalence class of $V(G_i)\setminus\{w_i\}$.

\begin{proposition}\label{pro:block_whiskers} Let $G$ be an accessible graph and let $B$ be a block of $G$. The graph $\overbar{B}$ constructed as above is accessible.
\end{proposition}
\begin{proof} 
Let $\pi: V(G) \to V(G)/\sim$ be the canonical projection. Let $T \in \mathcal{C}(\overbar{B})$. By construction, for any $i=1,\ldots, r$ $f_i$ is a free vertex of $\overbar{B}$, hence $T\subset V(B)$. Denote by $\overbar{\pi}$ the restriction of $\pi$ to $V(G)\setminus T$. We prove that $\overbar{\pi}$ induces a bijection between the connected components of $G\setminus T$ and the ones of $\overbar{B}\setminus T$.

Let $A$ be a connected component of $G\setminus T$. For any $i=1,\ldots, r$, let $G_i$ be the connected component of $G \setminus W$, where $W$ is the set of all the cutpoints of $\overbar{B}$. Let $a,b \in A$, and $a , a_1, \dots, a_\ell, b$ be a path in $V(G)\setminus T$ from $a$ to $b$. If $a$ and $b$ belong to the same $G_i$, then $\overbar{\pi}(a) = \overbar{\pi}(a_j) = \overbar{\pi}(b) = f_i$, for all $j =1, \dots, \ell$. Therefore, they are obviously connected in $\overbar{B}\setminus T$. If $a \in B$, and $b \in G_i$, then there exists $j$ such that $a_j, \dots, a_{\ell} \in G_i \cup \{w_i\}$ with, in particular, $a_j = w_i$. Then $\overbar{\pi}(a) = a, \overbar{\pi}(a_1) = a_1, \dots, \overbar{\pi}(a_{j-1})=a_{j-1}, f_i$ is a path from $\overbar{\pi}(a)$ and $\overbar{\pi}(b)=f_i$. The other cases follow by the same argument. 
Therefore, if $A$ is a connected component of $G\setminus T$, then $\overbar{\pi}(A)$ is a connected component of $\overbar{B}\setminus T$.

Let $D$ be a connected component of $\overbar{B}\setminus T$. Let $c,d \in D$ and let $c,u_1,\dots, u_{\ell},d$ be a path in $D$ from $c$ to $d$. Note that, by the definitions of path and $\overbar{B}$, for $i=1, \dots, \ell$, $u_i \in V(B)\setminus T$, that is $\overbar{\pi}^{-1}(u_i)=u_i$. If $c = f_j$ (resp. $d = f_j$) for some $j=1,\dots,r$, then set $\overbar{\pi}^{-1}(c) = v$ (resp. $\overbar{\pi}^{-1}(d) = v$), where $v \in V(H_j)$ and $\{w_j, v \}\in E(G)$. Otherwise, $\overbar{\pi}^{-1}(c) = c$ (resp. $\overbar{\pi}^{-1}(d) = d$). Then, $\overbar{\pi}^{-1}(c), u_1,\dots, u_{\ell},\overbar{\pi}^{-1}(d)$ is a path in $V(G)\setminus T$. It follows that if $D$ is a connected component of $\overbar{B}\setminus T$, then $\left(D \setminus \{f_j\}_{j \in J}\right) \cup \bigcup_{j\in J}G_j$ is a connected component of $G\setminus T$, where $J$ is the set of indices such that $f_j\in D$. 

The bijection between the connected components of $G\setminus T$ and the ones of $\overbar{B}\setminus T$ implies $c_G(T)=c_{\overbar{B}}(T)$. Since $J_G$ is unmixed by hypothesis, then $J_{\overbar{B}}$ is unmixed, as well. Moreover, if $T \in \mathcal{C}(\overbar{B})$, then $T \in \mathcal{C}(G)$. Due to the accessibility of $G$, there exists a vertex $a$ such that $T\setminus \{a\}\subset V(B)$ is a cutset of $G$ and so, using the bijection, $T\setminus \{a\}$ is a cutset of $\overbar{B}$, that $\overbar{B}$ is accessible.
\end{proof}

A block with a fixed number of vertices, say $n$, and minimum number of edges is a cycle $C_n$. 
It is useful to connect the  degree of the vertices with the cycle rank.

\begin{lemma}\label{lem:cyclerank}
    Let $G$ be a connected graph. The cycle rank of $G$ is 
    \[
    m(G) = 1+\frac{\sum_{v\in V(G)} (\deg v - 2)}{2}.
    \]
\end{lemma}
\begin{proof}
    From  (\cite[Theorem 4.5(a)]{HARARY}), we know $m(G)=q-p+1$ where $q=|E(G)|$ and $p=|V(G)|$. We can see
\[
    2q = \sum_{v\in V(G)} \deg v \quad \quad \mbox{and} \quad \quad p = \sum_{v\in V(G)} 1.
\]
So, we conclude that
\[
    m(G) = q-p+1=  1+\frac{\sum_{v\in V(G)} (\deg v - 2)}{2}.
\]
\end{proof}

By the previous lemma, we observe that fixed a cycle rank of $G$ the number of vertices with degree greater than 2 is bounded, but we do not have any information on the number of vertices $v$ with $\deg v\leq 2$. We will show that under the hypothesis of accessibility this cardinality is bounded, too.

Now we are going to state some general results for accessible blocks that we are going to exploit for the classification of accessible graphs with cycle rank 3 and in Section \ref{sec:chaincyc}. Let us introduce some notation. 
\begin{definition}
Given a block $B$, we say that a vertex $v\in V(B)$ is \textit{pivotal} if $\deg{v}\geq 3$.
\end{definition} 
\begin{definition}
Let $B$ be a block and let $a,b \in V(B)$ be two pivotal vertices. A path $L_i$ of length $i$ from $a$ to $b$ and such that any $v\in V(L_i)\setminus \{a,b\}$ is not pivotal is said a \textit{line} from $a$ to $b$.
\end{definition}

\begin{lemma}\label{lem:P2_P3}
Let $G$ be an accessible graph and $B$ a block of $G$. If two pivotal vertices $a,b$ of $B$ are connected by a line $L_i$, with $i\geq 2$, then $a$ is a cutpoint in $\overbar{B}$ and $b$ is not. Moreover, the following conditions hold:
\begin{enumerate}
\item $i<4$;
\item if $i=3$, there exists a unique vertex $c \in V(L_i)\setminus \{a,b\}$ which is a cutpoint in $\overbar{B}$. In particular, $c$ is such that $\{a,c\}\in E(G)$;
\item if $m(G) \geq 3$, there are no other lines $L_j$ from $a$ to $b$, with $j \in \{2,3\}$.
\end{enumerate}
\end{lemma}
\begin{proof}
By Proposition \ref{pro:block_whiskers}, we can focus on the graph $\overbar{B}$ which is accessible, too. 

Let $a$ and $b$ be two pivotal vertices of $B$ connected by a line $L_i$, with $i\geq 2$. We observe that $T=\{a,b\}$ is a cutset of $B$, and hence of $\overbar{B}$. In fact, $B\setminus T$ consists of at least two connected components: $L_i \setminus \{a,b\}$ and $B \setminus L_i$. Since $\overbar{B}$ is accessible, at least one between $a$ and $b$ has to be a cutpoint, assume $a$. Namely, there is a whisker $\{a,f\} \in E(\overbar{B})$. Moreover, at most one of them is a cutpoint, otherwise there should be another whisker $\{b,f'\}$ and $c_{\overbar{B}}(T)=4$, namely $\{f\}$, $\{f'\}$, $L_i\setminus \{a,b\}$ and $\overbar{B}\setminus (L_i\cup\{f,f'\})$.

From now on, we assume that $a$ is a cutpoint in $\overbar{B}$, while $b$ is not.

(1) Let $L_i = a, a_1, \cdots, a_{i-1}, b $ be a line from $a$ to $b$. Assume $i \geq 4$. $T = \{a,a_2\} \in \mathcal{C}(\overbar{B})$ and using the same argument of above, $a_2$ is not a cutpoint and $\overbar{B} \setminus T$ consists of three connected components: $\{f\}$, $\{a_1\}$ and $\overbar{B}\setminus ( T \cup \{a_1\})$. At the same time, $T'=\{a_2,b\} \in \mathcal{C}(\overbar{B})$ but it induces only two connected components: $\{a_3, \dots, a_{i-1}\}$ and $\overbar{B}\setminus ( L_i \setminus \{a,a_1\})$, which is a contradiction. 

(2) Let $i=3$ and $L_3 = a, a_1, a_2, b$ be a line from $a$ to $b$. Since $T=\{a_1, b\} \in \mathcal{C}(\overbar{B})$, $\overbar{B}$ is accessible and $b$ is not a cutpoint of $\overbar{B}$, then $a_1$ is a cutpoint of $\overbar{B}$. Moreover, since $T'=\{a,a_2\}\in \mathcal{C}(\overbar{B})$, then $a_2$ is not a cutpoint otherwise, $c_{\overbar{B}}(T)=4$.

(3) Suppose there are two lines $L_j'\neq L_i$, with $i,j \in \{2,3\}$, from $a$ to $b$. Consider the cutset $T=\{a,b\}$. Then, $\overbar{B}\setminus T$ consists of at least 4 connected components: $\{f\}$, $L_i \setminus \{a,b\}$, $L_j' \setminus \{a,b\}$, and $\overbar{B}\setminus (L_i \cup L_j')$, which is a contradiction.

\end{proof}

\begin{lemma}\label{lem:P_3}
    Let $G$ be an accessible graph and $B$ a block of $G$. If two pivotal vertices $a,b$ of $B$ are connected by a line $L_3$, then $\{a,b\}\in E(B)$.
\end{lemma}
\begin{proof}
    It is sufficient to show that the vertices $a$ and $b$ are not separable. By Lemma \ref{lem:P2_P3}, $a$ is a cutpoint in $\overbar{B}$ and let $\{a,f\} \in E(\overbar{B})$ be the whisker on $a$. Then,
    \[
        G\setminus \{a,b\} = \{f\} \sqcup (L_3\setminus \{a,b\}) \sqcup H,
    \]
where $H$ is a non-empty connected component of $G\setminus \{a,b\}$.
Assume by contradiction that $a$ and $b$ are separable. Let $L_3 = a,a_1,a_2,b$ be a line from $a$ to $b$ and let $T$ be a minimal cutset that separates $a$ and $b$. $T$ has vertices in $L_3\setminus \{a,b\}$ and in $H$. If $a_1\in T$, then $T'=\left(T\setminus\{a_1\}\right)\cup \{a_2\}$ is a cutset, as well. By Lemma \ref{lem:P2_P3} (2), $a_1$ is a cutpoint, but $a_2$ is not. Therefore, $|T| = |T'|$ but $c(T) = c(T')+1$, which is a contradiction.
\end{proof}

As an application, by means of the implementation described in Section \ref{sec: computation}, we will prove that the accessible blocks with whiskers of cycle rank 3 are the ones in Figures \ref{fig:chainsCR3} and \ref{fig:K4CR3}. From Lemma \ref{lem:cyclerank}, we have a bound on the number of pivotal vertices and, when $m(G) = 3$, it holds 
\[
\sum_{v \text{ pivotal vertices of } G} \left(\deg v - 2 \right)= 2\left(m(G) - 1\right) = 4. 
\]
All of the possible blocks with cycle rank 3 are showed in Figure \ref{Fig:ClassesCR3}, where the dot points denote pivotal vertices $v$, the number is $\deg v -2$ and the dashed line represents a line from a pivotal vertex to another.
As regards accessible graphs $\overbar{B}$ with $m(\overbar{B}) = 3$, they are obtained from the blocks $B$ in Figure \ref{Fig:ClassesCR3} by adding opportune whiskers. By Lemma \ref{lem:P2_P3}, there are no accessible graphs obtained from the blocks in the class of Figure \ref{Fig:ClassesCR3} (A). 
In Figures \ref{fig:chainsCR3} and \ref{fig:K4CR3}, all the accessible graphs $\overbar{B}$ with $m(\overbar{B}) = 3$ are displayed. As regards Figure   \ref{fig:chainsCR3}, the graphs (1)--(4) are obtained from the ones in Figure \ref{Fig:ClassesCR3} (B), while the graph (5) from the ones in Figure \ref{Fig:ClassesCR3} (C). These five graphs are chain of cycles that we characterize in the next section. Finally, the graphs in Figure \ref{fig:K4CR3} are all obtained from the blocks in Figure \ref{Fig:ClassesCR3} (D). In particular, they are obtained by the complete graph $K_4$ substituting any edge by a line $L_i$, with $i \in {1,2,3}$, and by adding whiskers in order to have accessibility of the graph. We denote this class of graphs by $\mathcal{K}_4$.

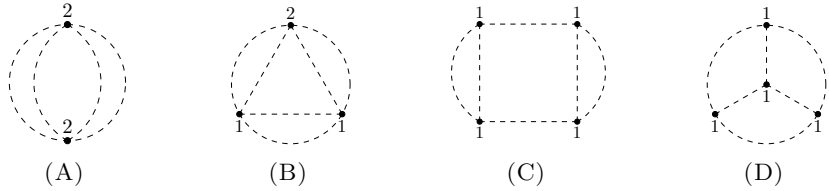
\begin{figure}[H]
 \centering
 \begin{subfigure}[t]{0.22\textwidth}
  \centering
  \resizebox{!}{0.7\textwidth}{
\begin{tikzpicture}
\draw[dashed] (0,1) circle (1cm);
\filldraw (0,0) circle (\rad) node [anchor=south]{$2$};
\filldraw (0,2) circle (\rad) node [anchor=south]{$2$};
\draw[dashed] (0,2) arc (120:240:1.154cm);
\draw[dashed] (0,0) arc (-60:60:1.154cm);
\end{tikzpicture}}
\caption{}\label{fig:ClassA}
\end{subfigure}%
\begin{subfigure}[t]{0.25\textwidth}
\centering 
  \resizebox{!}{0.6\textwidth}{
\begin{tikzpicture}
\draw[dashed] (0,0.577) circle (1.154cm);
\filldraw (-1,0) circle (\rad) node [anchor=north]{$1$};
\filldraw (1,0) circle (\rad) node [anchor=north]{$1$};
\filldraw (0,1.732) circle (\rad) node [anchor=south]{$2$};
\draw[dashed] (1,0)--(0,1.732);
\draw[dashed] (-1,0)--(0,1.732);
\draw[dashed] (1,0)--(-1,0);
\end{tikzpicture}}
\caption{}\label{fig:ClassB}
\end{subfigure}%
 \begin{subfigure}[t]{0.25\textwidth}
  \centering
 \resizebox{!}{0.6\textwidth}{
\begin{tikzpicture}

\filldraw (-1,0) circle (\rad) node [anchor=north]{$1$};
\filldraw (1,0) circle (\rad) node [anchor=north]{$1$};
\filldraw (-1,2) circle (\rad) node [anchor=south]{$1$};
\filldraw (1,2) circle (\rad) node [anchor=south]{$1$};
\draw[dashed] (1,0)--(-1,0);
\draw[dashed] (1,2)--(-1,2);
\draw[dashed] (1,0)--(1,2);
\draw[dashed] (-1,0)--(-1,2);

\draw[dashed] (-1,2) arc (120:240:1.154cm);
\draw[dashed] (1,2) arc (60:-60:1.154cm);
\end{tikzpicture}}
\caption{}\label{fig:ClassC}
\end{subfigure}%
\begin{subfigure}[t]{0.25\textwidth}
\centering
 \resizebox{!}{0.6\textwidth}{
\begin{tikzpicture}
\draw[dashed] (0,0.577 ) circle (1.154cm);
\draw[dashed] (1,0 )--(0,0.577 );
\draw[dashed] (0,0.577 )--(0,1.732 );
\draw[dashed] (0,0.577 )--(-1,0 );
\filldraw (0,0.577 ) circle (\rad) node [anchor=north]{$1$};
\filldraw (-1,0 ) circle (\rad) node [anchor=north]{$1$};
\filldraw (1,0 ) circle (\rad) node [anchor=north]{$1$};
\filldraw (0,1.732 ) circle (\rad) node [anchor=south]{$1$};
\end{tikzpicture}}
\caption{}\label{fig:ClassD}
\end{subfigure}\\
\caption{All classes of blocks having cycle rank 3.}\label{Fig:ClassesCR3}
\end{figure}

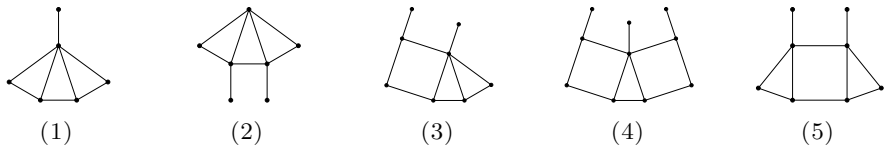
\begin{figure}[H]
\renewcommand{\thesubfigure}{\arabic{subfigure}}
 \centering
  \begin{subfigure}[t]{0.2\textwidth}
  \centering
  \resizebox{!}{0.5\textwidth}{
\begin{tikzpicture}
\filldraw (1-0.86,0.5) circle (\rad);
\filldraw ( 1,0) circle (\rad);
\filldraw ( 2,0) circle (\rad);
\filldraw (2.86,0.5) circle (\rad);
\filldraw (1.5,1.5) circle (\rad);

\draw (1-0.86,0.5)--(1.5,1.5);
\draw (1-0.86,0.5)--(1,0);
\draw (1.5,1.5)--(1,0);
\draw (2,0)--(1,0);
\draw (1.5,1.5)--(2,0);
\draw (2,0)--(2.86,0.5);
\draw (1.5,1.5)--(2.86,0.5);

\filldraw (1.5,2.5) circle (\rad);
\draw (1.5,1.5)--(1.5,2.5);
\end{tikzpicture}}
\caption{}\label{fig:Chain1}
\end{subfigure}%
\begin{subfigure}[t]{0.2\textwidth}
\centering 
   \resizebox{!}{0.5\textwidth}{
\begin{tikzpicture}
\filldraw (1-0.86,0.5) circle (\rad);
\filldraw ( 1,0) circle (\rad);
\filldraw ( 2,0) circle (\rad);
\filldraw (2.86,0.5) circle (\rad);
\filldraw (1.5,1.5) circle (\rad);

\draw (1-0.86,0.5)--(1.5,1.5);
\draw (1-0.86,0.5)--(1,0);
\draw (1.5,1.5)--(1,0);
\draw (2,0)--(1,0);
\draw (1.5,1.5)--(2,0);
\draw (2,0)--(2.86,0.5);
\draw (1.5,1.5)--(2.86,0.5);

\filldraw (1,-1) circle (\rad);
\filldraw (2,-1) circle (\rad);
\draw (1,-1)--(1,0);
\draw (2,-1)--(2,0);
\end{tikzpicture}}
\caption{}\label{fig:Chain2}
\end{subfigure}%
\centering
 \begin{subfigure}[t]{0.2\textwidth}
  \centering
   \resizebox{!}{0.5\textwidth}{
\begin{tikzpicture}
\filldraw (1-0.951*1.581,0.309*1.581) circle (\rad);
\filldraw (1-0.951*1.581+0.5,0.309*1.581+1.5) circle (\rad);
\filldraw ( 1,0) circle (\rad);
\filldraw ( 2,0) circle (\rad);
\filldraw (2.86,0.5) circle (\rad);
\filldraw (1.5,1.5) circle (\rad);

\draw (1-0.951*1.581,0.309*1.581)-- (1-0.951*1.581+0.5,0.309*1.581+1.5);
\draw (1-0.951*1.581,0.309*1.581)--(1,0);
\draw (1.5,1.5)--(1-0.951*1.581+0.5,0.309*1.581+1.5);
\draw (1.5,1.5)--(1,0);
\draw (2,0)--(1,0);
\draw (1.5,1.5)--(2,0);
\draw (2,0)--(2.86,0.5);
\draw (1.5,1.5)--(2.86,0.5);

\filldraw (1.816,2.448) circle (\rad);
\filldraw (1-0.951*1.581+0.5+0.316,0.309*1.581+1.5+0.948) circle (\rad);
\draw  (1.816,2.448)--(1.5,1.5);
\draw (1-0.951*1.581+0.5+0.316,0.309*1.581+1.5+0.948)--(1-0.951*1.581+0.5,0.309*1.581+1.5);
\end{tikzpicture}}
\caption{}\label{fig:Chain3}
\end{subfigure}%
\begin{subfigure}[t]{0.2\textwidth}
\centering
 \resizebox{!}{0.5\textwidth}{
\begin{tikzpicture}
\filldraw (1-0.951*1.581,0.309*1.581) circle (\rad);
\filldraw (1-0.951*1.581+0.5,0.309*1.581+1.5) circle (\rad);
\filldraw (2+0.951*1.581,0.309*1.581) circle (\rad);
\filldraw (2+0.951*1.581-0.5,0.309*1.581+1.5) circle (\rad);
\filldraw ( 1,0) circle (\rad);
\filldraw ( 2,0) circle (\rad);
\filldraw (1.5,1.5) circle (\rad);

\draw (1-0.951*1.581,0.309*1.581)-- (1-0.951*1.581+0.5,0.309*1.581+1.5);
\draw (1-0.951*1.581,0.309*1.581)--(1,0);
\draw (1.5,1.5)--(1-0.951*1.581+0.5,0.309*1.581+1.5);
\draw (1.5,1.5)--(1,0);
\draw (2,0)--(1,0);
\draw (1.5,1.5)--(2,0);
\draw (2,0)--(2+0.951*1.581,0.309*1.581);
\draw (2+0.951*1.581-0.5,0.309*1.581+1.5) --(2+0.951*1.581,0.309*1.581);
\draw (1.5,1.5)--(2+0.951*1.581-0.5,0.309*1.581+1.5) ;

\filldraw (1.5,2.5) circle (\rad);
\filldraw (2+0.951*1.581-0.5-0.316,0.309*1.581+2.448) circle (\rad);
\filldraw (1-0.951*1.581+0.5+0.316,0.309*1.581+1.5+0.948) circle (\rad);
\draw (1-0.951*1.581+0.5+0.316,0.309*1.581+1.5+0.948)--(1-0.951*1.581+0.5,0.309*1.581+1.5);
\draw  (1.5,2.5)--(1.5,1.5);
\draw (2+0.951*1.581-0.5-0.316,0.309*1.581+1.5+0.948)--(2+0.951*1.581-0.5,0.309*1.581+1.5);
\end{tikzpicture}}
\caption{}\label{fig:Chain4}
\end{subfigure}%
\begin{subfigure}[t]{0.2\textwidth}
\centering
 \resizebox{!}{0.5\textwidth}{
\begin{tikzpicture}
\filldraw (1,0) circle (\rad);
\filldraw (2.5,0) circle (\rad);
\filldraw (1,1.5) circle (\rad);
\filldraw (2.5,1.5) circle (\rad);
\filldraw (1-0.942,0.333) circle (\rad);
\filldraw (2.5+0.942,0.333) circle (\rad);

\draw(1,0)--(2.5,0);
\draw(1,1.5)--(2.5,1.5);
\draw(1,0)--(1,1.5);
\draw(2.5,0)--(2.5,1.5);

\draw(1,0)--(1-0.942,0.333);
\draw(1,1.5)--(1-0.942,0.333);
\draw(2.5,0)--(2.5+0.942,0.333) ;
\draw(2.5+0.942,0.333) --(2.5,1.5);

\filldraw (1,2.5) circle (\rad);
\filldraw (2.5,2.5) circle (\rad);
\draw(1,1.5)--(1,2.5);
\draw(2.5,1.5)--(2.5,2.5);
\end{tikzpicture}}
\caption{}\label{fig:Chain5}
\end{subfigure}\\
\caption{The accessible chains of cycles with cycle rank 3.}\label{fig:chainsCR3}
\end{figure}

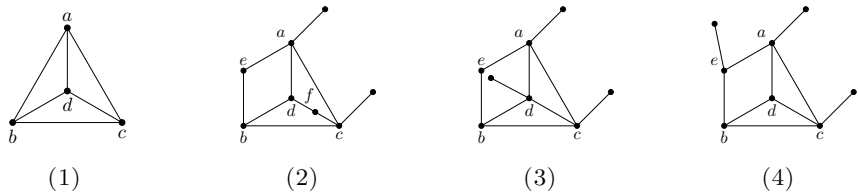
\begin{figure}[H]
\renewcommand{\thesubfigure}{\arabic{subfigure}}
\centering
\begin{subfigure}[t]{0.25\textwidth}
\centering
 \resizebox{!}{0.6\textwidth}{
\begin{tikzpicture}
\filldraw (0,0.577) circle (\rad) node [anchor=north]{$d$};
\filldraw (-1,0) circle (\rad) node [anchor=north]{$b$};
\filldraw (1,0) circle (\rad) node [anchor=north]{$c$};
\filldraw (0,1.732) circle (\rad) node [anchor=south]{$a$};
\draw (1,0)--(0,0.577);
\draw (0,0.577)--(0,1.732);
\draw (0,0.577)--(-1,0);
\draw (-1,0)--(1,0);
\draw (-1,0)--(0,1.732);
\draw (1,0)--(0,1.732);
\end{tikzpicture}}
\caption{}\label{fig:K4.1}
\end{subfigure}%
\begin{subfigure}[t]{0.25\textwidth}
\centering
 \resizebox{!}{0.6\textwidth}{
\begin{tikzpicture}
\filldraw (0,0.577) circle (\rad) node [anchor=north]{$d$};
\filldraw (-1,0) circle (\rad) node [anchor=north]{$b$};
\filldraw (-1,1.155) circle (\rad) node [anchor=south]{$e$};
\filldraw (0.5,0.288) circle (\rad);
\filldraw (1,0) circle (\rad) node [anchor=north]{$c$};
\filldraw (0,1.732) circle (\rad)node [anchor= south east]{$a$};
\draw (1,0)--(0,0.577);
\draw (0,0.577)--(0,1.732);
\draw (0,0.577)--(-1,0);
\draw(-1,1.155)--(-1,0);
\draw(-1,1.155)--(0,1.732);
\draw (0,1.732)--(1,0);
\draw (-1,0)--(1,0);
\node at (0.5-0.13,0.288+0.35) {$f$};

\filldraw (0.707,1.732+0.707) circle (\rad);
\filldraw (1.707,0.707) circle (\rad) ;
\draw (1,0)--(1.707,0.707);
\draw (0.707,1.732+0.707)--(0,1.732);
\end{tikzpicture}}
\caption{}\label{fig:K4.2}
\end{subfigure}%
\begin{subfigure}[t]{0.25\textwidth}
\centering
 \resizebox{!}{0.6\textwidth}{
\begin{tikzpicture}
\filldraw (0,0.577) circle (\rad) node [anchor=north]{$d$};
\filldraw (-1,0) circle (\rad) node [anchor=north]{$b$};
\filldraw (-1,1.155) circle (\rad) node [anchor=south]{$e$};
\filldraw (1,0) circle (\rad) node [anchor=north]{$c$};
\filldraw (0,1.732) circle (\rad)node [anchor= south east]{$a$};
\draw (1,0)--(0,0.577);
\draw (0,0.577)--(0,1.732);
\draw (0,0.577)--(-1,0);
\draw(-1,1.155)--(-1,0);
\draw(-1,1.155)--(0,1.732);
\draw (-1,0)--(1,0);
\draw (1,0)--(0,1.732);

\filldraw (0.707,1.732+0.707) circle (\rad);
\filldraw (1.707,0.707) circle (\rad) ;
\filldraw (-0.8, 1) circle (\rad) ;
\draw (1,0)--(1.707,0.707);
\draw (0.707,1.732+0.707)--(0,1.732);
\draw (-0.8, 1)--(0,0.577);
\end{tikzpicture}}
\caption{}\label{fig:K4.3}
\end{subfigure}%
\begin{subfigure}[t]{0.25\textwidth}
\centering
 \resizebox{!}{0.6\textwidth}{
\begin{tikzpicture}
\filldraw (0,0.577) circle (\rad) node [anchor=north]{$d$};
\filldraw (-1,0) circle (\rad) node [anchor=north]{$b$};
\filldraw (-1,1.155) circle (\rad) node [anchor=south east]{$e$};
\filldraw (1,0) circle (\rad) node [anchor=north]{$c$};
\filldraw (0,1.732) circle (\rad)node [anchor= south east]{$a$};
\draw (1,0)--(0,0.577);
\draw (0,0.577)--(0,1.732);
\draw (0,0.577)--(-1,0);
\draw(-1,1.155)--(-1,0);
\draw(-1,1.155)--(0,1.732);
\draw (-1,0)--(1,0);
\draw (1,0)--(0,1.732);

\filldraw (0.707,1.732+0.707) circle (\rad);
\filldraw (1.707,0.707) circle (\rad) ;
\filldraw (-1.2, 1.155+0.98) circle (\rad) ;
\draw (1,0)--(1.707,0.707);
\draw (0.707,1.732+0.707)--(0,1.732);
\draw  (-1.2,1.155+0.98)--(-1,1.155);
\end{tikzpicture}}
\caption{}\label{fig:K4.4}
\end{subfigure}\\
\caption{The class $\mathcal{K}_4$.}\label{fig:K4CR3}
\end{figure}

In the next results, by focusing on the lines connecting two pivotal vertices, we exhibit that, starting from blocks belong to the class (D) of Figure \ref{Fig:ClassesCR3}, there are no other possible accessible blocks with whiskers than the graphs (1)--(4) in Figure \ref{fig:K4CR3}.

\begin{lemma}\label{lem:K4}
Let $\overbar{B}$ be an accessible graph such that $B$ is a block with $m(B)= 3$ that belongs to the class (D) of Figure \ref{Fig:ClassesCR3}. Then in $B$ there are at most two lines $L_2$ that have no vertex in common and there is no line $L_3$.
\end{lemma}
\begin{proof}
 Let $a,b,c,d$ be the pivotal vertices of $\overbar{B}$. Without loss of generality, assume that there are two lines $L_2$ in $B$ having a vertex in common: one from $a$ to $b$ and a second one from $a$ to $c$. We claim that $a$ has a whisker in $\overbar{B}$ and $b$ and $c$ have no whiskers. In fact, $\{a,b\}$ and $\{a,c\}\in \CC(\overbar{B})$. By Lemma \ref{lem:P2_P3}, either $a$ has a whisker or both $b$ and $c$ have whiskers. Moreover, $T=\{a,b,c\}\in\CC(\overbar{B})$ and if $b$ and $c$ have whiskers $c(T)=5$. Hence the claim follows.
 
 Let $a_1$ (resp. $a_1'$) be the vertex of degree $2$ in the line $L_2$ from $a$ to $b$ (resp. to $c$). Let $T'=\{c,d,a_1\}\in \CC(\overbar{B})$ and $T''=\{b,d,a_1'\}\in \CC(\overbar{B})$. We observe that there are no subsets of $T'$ (resp. $T''$) disconnecting the block. Hence $d$, $a_1$ and $a_1'$ have whiskers. But, for $T'''=\{d,a_1,a_1'\}\in \CC(\overbar{B})$, it holds $c(T''')=5$, which is a contradiction.  
 
 Finally, suppose by contradiction that we have a line $L_3$ from $a$ to $b$. By Lemma \ref{lem:P_3}, $\{a,b\}\in E(G)$. This implies that the cycle rank of $G$ is greater than $3$.
\end{proof}

\begin{corollary}
The accessible graphs $\overbar{B}$ such that $B$ belongs to the class in Figure \ref{Fig:ClassesCR3} (D) are all and only the graphs in $\mathcal{K}_4$ displayed in Figure \ref{fig:K4CR3}.
\end{corollary}
\begin{proof}

If $B$ has no line $L_2$, then $\overbar{B}$ is a $K_4$ with or without whiskers (Figure \ref{fig:K4CR3} (1)).
 
If $B$ has 2 lines $L_2$, $\overbar{B}$ is a bipartite graph and the only accessible bipartite graph with cycle rank 3 is the one in Figure \ref{fig:K4CR3} (2).
 
 Suppose $B$ has exactly one line $L_2$. Assume it is from $a$ to $b$ and denote by $e$ the unique vertex of degree $2$ in $L_2$. Let $c$ and $d$ be the other 2 pivotal vertices. We observe that the non-empty cutsets of $B$ are $\{a,b\}$ and $\{c,d,e\}$. By Lemma \ref{lem:P2_P3}, without loss of generality, we may assume that $a$ has a whisker and $b$ has no whisker.  Since $\{c,d,e\}$ has cardinality $3$ and none of its subsets is a cutset of the block, we have that exactly $2$ vertices in $\{c,d,e\}$ have a whisker. That is either both $c$ and $d$ have a whisker, or one whisker is on $e$ and the other one is, without loss of generality, on $c$. Then the obtained $\overbar{B}$ are the non-bipartite and non-complete graphs (3) and (4) in Figure \ref{fig:K4CR3}.
 
\end{proof}



\section{Chain of cycles}\label{sec:chaincyc}

In this section, we define a new family of graphs, the chain of cycles, and we classify the ones with Cohen-Macaulay binomial edge ideal by means of combinatorial properties. 

Given a graph $G$, we denote by $G_v$ the graph obtained from $G$ by adding edges $\{u,w\}$ to $E(G)$ for all $u,w \in V(G)$ adjacent to $v$. We recall the following definition given first in \cite{BMS2}.
\begin{definition}
Let $G$ be a graph. $J_G$ is \textit{strongly unmixed} if the connected components of $G$ are complete graphs or if $J_G$ is unmixed and there exists a cutpoint $v$ of $G$ such that $J_{G\setminus \{v\}}$, $J_{G_v}$ and $J_{G_{v}\setminus \{v\}}$ are strongly unmixed.
\end{definition}

\begin{definition}\label{def:chaincycle}
 Let $B$ be a block with $m(B) = r$ such that $B=\bigcup_{i=1}^r D_i$ where $D_i$ are cycles, and if $j=i+1$ then $E(D_i)\cap E(D_j)=E(P)$, where $P$ is a path, otherwise $E(D_i)\cap E(D_j)=\emptyset$. We call $B$ a \textit{chain of cycles}.

\end{definition}

\begin{lemma}\label{lem:C3C4} Let $\overbar{B}$ be an accessible graph such that $B = \bigcup_{i=1}^r D_i$ is a chain of cycles. Then $D_i\in\{C_3,C_4\}$ and $E(D_i)\cap E(D_{i+1})$ is an edge of $B$.  
\end{lemma}

\begin{proof}
 If $r\in \{1,2\}$, the claim follows by \cite{R2}. From now on, assume $r\geq 3$, that is $m(B) \geq 3$.
 
 Let $i=1$, and let $a,b \in V(D_1)\cap V(D_2)$ be pivotal vertices of $B$. By Lemma \ref{lem:P2_P3}, there is a unique line $L_i$, with $i \in \{2,3\}$, from $a$ to $b$. Hence, we may assume $E(D_1)\cap E(D_2)$ is an edge and $D_1$ is either $C_3$ or $C_4$. By the same argument, $D_r$ has the same property.

 Let $i\in\{2,\ldots,r-1\}$ and let $a,b \in V(D_{i})\cap V(D_{i+1})$ be pivotal vertices of $B$. $T=\{a,b\}$ is a cutset of $\overbar{B}$ and since $\overbar{B}$ is accessible, either $a$ or $b$ is a cutpoint in $\overbar{B}$. Therefore, $E(D_i)\cap E(D_{i+1})$ is an edge, due to the unmixedness of $J_{\overbar{B}}$.
 
 Let $a,b \in V(D_{i-1})\cap V(D_i)$ and $c,d \in V(D_{i})\cap V(D_{i+1})$ be pivotal vertices of $B$. Let $T =\{a,b\}$ and $T'=\{c,d\}$. Assume that $c\not\in T$ and, without loss of generality, $L_j$ is a line from $a$ to $c$. We will prove that $j=1$. By contradiction, suppose $j>1$. Hence $T''=\{a,c\}$ is a cutset. By Lemma \ref{lem:P2_P3} applied to $T$, in $T'$ and $T''$ there are two distinct vertices $u,v\in\{a,b,c\}$ that are cutpoints. We obtain a contradiction since $c_{\overbar{B}}(\{u,v\})=4$.
 
It follows that $\{a,c\}$ is an edge and either $b=d$ or $\{b,d\}$ is an edge. That is $D_i$ is either $C_3$ or $C_4$.
\end{proof}

\begin{remark}
\label{rmk:aibi}
By Lemma \ref{lem:C3C4} we can relabel the vertices of $B$ so that $V(D_i)\cap V(D_{i+1})=\{w_i,u_i\}$ and such that if $w_i\neq w_{i+1}$  (resp. $u_i\neq u_{i+1}$) then the edge $\{w_i,w_{i+1}\}$ (resp. $\{u_i,u_{i+1}\}$)  belongs to $E(D_{i+1})$ and does not belong to any cycle $D_j$ for $j\neq i+1$. 
\end{remark}

\begin{lemma}\label{lem:cutchain}
Let $\overbar{B}$ be an accessible graph such that $B = \bigcup_{i=1}^r D_i$ is a chain of cycles. Following the labelling defined in Remark \ref{rmk:aibi}, every $w_i$ is a cutpoint in $\overbar{B}$ and $u_i$ is not a cutpoint in $\overbar{B}$.
\end{lemma}
\begin{proof}
 We observe that $\{w_1,u_1\}$ is a cutset of $\overbar{B}$. Hence, due to accessibility of $\overbar{B}$ either $w_1$ or $u_1$ is a cutpoint in $\overbar{B}$. Without loss of generality, we may assume $w_1$ is a cutpoint. We observe that also $\{u_1,w_2\}$,  $\{w_1,u_2\}$ are cutsets of $\overbar{B}$. Hence, $w_2$ must be a cutpoint and $u_2$ cannot be a cutpoint. Applying the same argument for all $\{w_i,u_i\}$, the assertion follows. 
\end{proof}

\begin{remark}\label{rmk:WU}
From now on, thanks to Lemma \ref{lem:P_3} and Lemma \ref{lem:cutchain},  we may consider the following partition of the set of vertices of $B$:
\[
V(B) = W \sqcup U,
\]
where $W$ consists of all the cutpoints of $\overbar{B}$, and $U=V(B)\setminus W$. We observe that the induced subgraphs on $W$ and $U$ (respectively) are paths. 
\end{remark}

\begin{lemma}\label{lem:C4nextC3}
Let $\overbar{B}$ be an accessible graph such that $B = \bigcup_{i=1}^r D_i$ is a chain of cycles. If $D_i=C_4$, then $D_{i+1}=C_3$.
\end{lemma}

\begin{proof}
By contradiction, suppose that $D_i$ and $D_{i+1}$ are both $C_4$. By Lemma \ref{lem:cutchain}, $w_{i-1}$, $w_i$ , $w_{i+1}$ are all cutpoints  while $u_{i-1},u_i,u_{i+1}$ are not cutpoints. We can see that $T=\{w_{i-1},u_i,w_{i+1}\}\in \CC(\overbar{B})$ and $c(T)=5$. Contradiction.
\end{proof}

\begin{lemma}\label{prop:center}
Let $\overbar{B}$ be an accessible graph such that $B = \bigcup_{i=1}^r D_i$ is a chain of cycles. Let $v\in V(B)$ with $\deg(v)\geq 5$ or $\deg(v)\geq 4$ if $v$ is a vertex of a $C_4$. Then $v$ is a cutpoint.
\end{lemma}
\begin{proof}
By hypothesis, we can identify $T_i=\{v,v_i\}\in \CC(B)$ for $i=1,2,3$, with $\{v_1,v_2\}$,$\{v_2,v_3\}\in E(B)$. Since $\overbar{B}$ is accessible, we obtain that each $T_i$ contains exactly a cutpoint. By contradiction, assume that $v$ is not a cutpoint. The latter implies that $v_1,v_2$ and $v_3$ belong to $W$, namely they are cutpoints in $\overbar{B}$. We observe that $T=\{v,v_1,v_3\}\in \CC(\overbar{B})$, but $c(T)=5$. Contradiction.
\end{proof}

\begin{remark}\label{rem:Gvw}
Let $G$ be a graph and let $v,w \in V(G)$ with $v \neq w$. Then $(G\setminus \{v\})_{w} = G_w \setminus \{v\}$. Clearly $V((G\setminus \{v\})_{w}) =V( G_w \setminus v)=V(G\setminus \{v\})$. We have:
\[
E(G_w\setminus \{v\})= (E(G) \cup \{\{x,y\} \ | \ x,y \in N_{G}(w)\}) \setminus \{\{v,u\} \ | \ u  \in N_{G_w}(v)\}.
\]
Moreover, we observe that $N_{G_w}(v)$ is either equal to $N_G(v)$ if $\{v,w\} \notin E(G)$ or to $N_{G}(v) \cup N_{G\setminus \{v\}}(w)$ if $\{v,w\} \in E(G)$, that is 
\begin{align*}
E(G_w\setminus \{v\}) &=(E(G) \setminus \{\{v,u\} \ | \ u  \in N_{G}(v)\})  \cup \{\{x,y\} \ | \ x,y \in N_{G\setminus \{v\}}(w)\} \\ &=E((G\setminus \{v\})_w).   
\end{align*}

\end{remark}

\begin{lemma}\label{lem:strong_Gminv}
Let $G$ be a graph such that $J_G$ is unmixed and let $v \in V(G)$ be a free vertex of $G$. If $J_{G \setminus \{v\}}$ is strongly unmixed, then $J_G$ is strongly unmixed. 
\end{lemma}

\begin{proof}
We proceed by induction on the number $r$ of cutpoints of $G \setminus \{v\}$. \\
If $r=0$, then $G \setminus \{v\}$ is a complete graph. The latter implies that $G$ is a complete graph with or without a whisker, and it is immediate to see that $J_G$ is strongly unmixed. \\
We assume $r > 0 $ and the thesis true for any graph $G\setminus \{v\}$ with at most $r-1$ cutpoints. Let $\{w\} \in \CC(G\setminus \{v\})$ such that the binomial edge ideals of 
$(G\setminus \{v\}) \setminus \{w\}, (G\setminus \{v\} )_w$, and $(G\setminus \{v\} )_w \setminus \{w\}$
are strongly unmixed. We observe that $w$ is also a cutpoint for $G$, otherwise $\{v,w\}$ is a cutset for $G$ contradicting the fact that $v$ is a free vertex. From Remark \ref{rem:Gvw}, it follows that $(G\setminus \{v\})_w= G_w \setminus \{v\}$ and $(G\setminus \{v\})_w \setminus \{w\} = G_w \setminus \{v,w\}$. The latter graphs and $G\setminus \{v,w\}$ are three graphs having a number of cutpoints less than or equal to $r-1$,  hence by the inductive hypothesis the assertion follows. 
\end{proof}

\begin{lemma}\label{lem:strong_unmixed_decomposable}
Let $G_1$ and $G_2$ be two graphs and let $G=G_1\cup G_2$ be such that $V(G_1)\cap V(G_2)=\{v\}$, with $v$ free vertex of $G_1$ and $G_2$. The following conditions are equivalent:
\begin{enumerate}
    \item $J_{G_1}$ and $J_{G_2}$ are strongly unmixed (resp. $G_1$ and $G_2$ are accessible);
    \item $J_{G}$ is strongly unmixed (resp. $G$ is accessible).
\end{enumerate}
\end{lemma}
\begin{proof}
With respect to accessibility the two conditions are equivalent by \cite[Proposition 2.6]{RR} and \cite[Lemma 2.3]{RR}. Now we focus on strong unmixedness.

(1)$\Rightarrow$(2). By \cite[Proposition 2.6]{RR}, $J_G$ is unmixed. Let $\{v_1,\ldots, v_r\}\subset V(G_1)$ such that $v_i$ is a cutpoint of $H_i=G_1\setminus\{v_1,\ldots,v_{i-1}\}$ and $J_{H_i}$ is strongly unmixed. 
Let $\{u_1,\ldots, u_s\}\subset V(G_2)$ be the set satisfying the same property for $G_2$. 

Since $v$ is a free vertex, it is neither a cutpoint of $G_1$ nor a cutpoint of $G_2$. Moreover, for any $a\in G_i$, $v$ remains a free vertex of $G_i\setminus \{a\}$.

We claim that $G$ is strongly unmixed with respect to the sequence of vertices \[v_1,\ldots,v_r,u_1,\ldots,u_s,\]
adding $v$ if necessary.

By \cite[Proposition 2.6]{RR},  $G\setminus\{v_1\}$ is decomposable in $H_1$ and $G_2$ whose ideals are both unmixed. Hence $J_{H_1\cup G_2}$ is unmixed, as well. By the same argument, we can remove the remaining vertices  $\{v_2,\ldots,v_r,u_1,\ldots,u_s\}$ obtaining unmixed ideals. Now either all the components are complete graphs or there is only one containing $v$ that is decomposable into 2 complete graphs. In this case, we add $v$ to the sequence of cutpoints.

(2)$\Rightarrow$(1). We proceed by induction on the number $r$ of cutpoints of $G$. We observe that $r\geq 1$ since $G$ is decomposable, hence we take $r=1$ as base case. In this case, $v$ is the unique cutpoint and $G_1\setminus \{v\}$ and $G_2 \setminus \{v\}$ are both complete graphs, that is $G_1$ and $G_2$ are complete graphs and the thesis follows. 
We assume $r > 1$ and that the thesis holds true for any number of cutpoints less than or equal to $r-1$. 
Let $w$ be a cutpoint of $G$. If $w =v$, then we obtain that $J_{G_1 \setminus \{v\}}$ and $J_{G_2\setminus \{v\}}$ are strongly unmixed, and since $v$ is a free vertex of $G_1$ and $G_2$, then the assertion follows from Lemma \ref{lem:strong_Gminv}. If $w \neq v$, we assume without loss of generality that $w \in V(G_1 \setminus \{v\})$. We obtain that $G \setminus \{w\}$ has two connected components, one $H=H_1 \cup G_2$ with $V(H_1)\cap V(G_2)=\{v\}$ and another component $H_2$. From the strong unmixedness of $J_{G \setminus \{w\}}$ and from the inductive hypothesis, we obtain that $J_{H_1}$, $J_{G_2}$ and $J_{H_2}$ are strongly unmixed and since $G_1 \setminus \{w\}= H_1 \cup H_2$, then $J_{G_1 \setminus \{w\}}$ is also strongly unmixed. By similar arguments, one can prove that also $J_{(G_1)_w}$ and $J_{(G_1)_w \setminus \{w\}}$ are strongly unmixed, that is $J_{G_1}$ is strongly unmixed.

\end{proof}

\begin{setup}\label{setup: chain}
Let $\overbar{B}$ be a block with whiskers, where $B=\bigcup_{i=1}^r D_i$ is a chain of cycles, satisfying the following properties:
\begin{enumerate}
 \item each $D_i\in \{C_3, C_4\}$;
 \item if $D_i=C_4$ then $D_{i+1}=C_3$;
 \item $E(D_i)\cap E(D_{i+1})=\{\{w_i,u_i\}\}$, where $w_i$ is a cutpoint and $u_i$ is not a cutpoint;
 \item $\{w_i,w_{i+1}\}\in E(D_{i+i})$ (resp. $\{u_i,u_{i+1}\}\in E(D_{i+1})$) or $w_i=w_{i+1}$ (resp. $u_i=u_{i+1}$);
 \item if $D_1=C_4$ with $V(D_1)=\{w_0,w_1,u_0,u_1\}$ with $\{w_0,w_1\},\{u_0,u_1\}\in E(D_1)$ then $w_0$ and $w_1$  are cutpoints, whereas $u_0$ and $u_1$ are not cutpoints;
 \item if $D_r=C_4$ with $V(D_r)=\{w_r,w_{r+1},u_r,u_{r+1}\}$ with $\{w_r,w_{r+1}\},\{u_r,u_{r+1}\}\in E(D_r)$ then $w_r$ and $w_{r+1}$ are cutpoints, whereas $u_r$ and $u_{r+1}$ are not cutpoints;
 \item if $v\in V(B)$ with $\deg(v)\geq 5$ or $\deg(v)\geq 4$ with $v$ a vertex of a $C_4$ then $v$ is a cutpoint.
 
 \end{enumerate}

\end{setup}

In Figure \ref{fig:chainsetup}, an example of a graph $\overbar{B}$ satisfying Setup \ref{setup: chain} is displayed.

\begin{figure}[H]
\centering
 \resizebox{!}{0.2\textwidth}{
\begin{tikzpicture}

\filldraw (0, 0) circle (\rad);
\filldraw (2, 0) circle (\rad);
\filldraw (3, 0) circle (\rad);
\filldraw (4, 0) circle (\rad);
\filldraw (6, 0) circle (\rad);
\filldraw (8, 0) circle (\rad);
\filldraw (10, 0) circle (\rad);
\filldraw (11, 0) circle (\rad);
\filldraw (12, 0) circle (\rad);
\filldraw (13, 0) circle (\rad);
\filldraw (14, 0) circle (\rad);

\filldraw (0, 2) circle (\rad);
\filldraw (3, 2) circle (\rad);
\filldraw (7, 2) circle (\rad);
\filldraw (9, 2) circle (\rad);
\filldraw (12, 2) circle (\rad);

\filldraw (0, 3.5) circle (\rad);
\filldraw (3, 3.5) circle (\rad);
\filldraw (7, 3.5) circle (\rad);
\filldraw (9, 3.5) circle (\rad);
\filldraw (12, 3.5) circle (\rad);

\draw (0,0)--(14,0)--(12,2)--(0,2)--cycle;

\draw (2,0)--(3,2);
\draw (3,0)--(3,2);
\draw (4,0)--(3,2);
\draw (6,0)--(7,2);
\draw (7,2)--(8,0);
\draw (8,0)--(9,2);
\draw (9,2)--(10,0);
\draw (11,0)--(12,2);
\draw (12,0)--(12,2);
\draw (13,0)--(12,2);

\draw (0, 2)--(0,3.5);
\draw (3, 2)--(3,3.5);
\draw (7, 2)--(7,3.5);
\draw (9, 2)--(9,3.5);
\draw (12, 2)--(12,3.5);

\end{tikzpicture}}
\caption{A graph $\overbar{B}$ satisfying Setup \ref{setup: chain}}\label{fig:chainsetup}
\end{figure}

\begin{lemma}\label{lem:compU}
Let $\overbar{B}$ be a graph satisfying Setup \ref{setup: chain}, and let $T\in \CC(\overbar{B})$. Then for all $u\in U\cap T$ there exists $w\in W\cap T$ such that $\{u,w\}\in \CC(\overbar{B})$.
\end{lemma}
\begin{proof}
By contradiction, assume that there exists $u \in T\cap U$ such that any vertex $w \in W$ for which $\{u,w\}\in \CC(\overbar{B}) $ does not belong to $T$.
Let $T'=T\setminus \{u\}$. We prove that $c_{\overbar{B}}(T)=c_{\overbar{B}}(T')$. Let $H$ be the connected component of $\overbar{B}\setminus T'$ containing $u$. We prove that $H\setminus u$ is connected. Let $v,v' \in V(H\setminus \{u\})$ and let $\pi: v, v_1, \ldots, v_{\ell}, v'$ be a path in $H$ from $v$ to $v'$. If $u \notin V(\pi)$, then $v$ and $v'$ are connected in $H\setminus \{u\}$ through $\pi$. If $u\in V(\pi)$, then $\pi: v, v_1, \ldots, v_{i-1}, u, v_{i+1}, \ldots, v_{\ell}, v'$.

We claim that there exists a path $v_{i-1}, z_1, \ldots, z_{m}, v_{i+1}$ with $\{u,z_j\}\in \CC(\overbar{B})$ and $z_j \notin T$ for any $j \in \{1,\ldots, m\}$. If $v_{i-1},v_{i+1} \in W$, then $v_{i-1}=w_{j}$, $v_{i+1}=w_{k}$ with $j < k$ as in the Setup \ref{setup: chain}, hence the vertices $w_{j+1},\ldots, w_{k-1}$ make a path between $w_j$ and $w_k$.
Furthermore being $w_j, w_k$ adjacent to $u$, then $\{u,w_{j+1}\}, \ldots, \{u,w_{k-1}\} \in \CC(\overbar{B})$ and in particular $w_{j+1},\ldots, w_{k-1} \notin T$. In this case, the claim follows.\\
Now, we deal with the case $v_{i-1}$ or $v_{i+1} \in U$. Observe that any vertex $u' \in U$ adjacent to $u$ is also adjacent to a vertex $w' \in W$ such that $\{u,w'\}\in \CC(\overbar{B})$. In fact, let $D_k$ be the cycle containing $u$ and $u'$. The vertex $w' \neq u$ adjacent to $u'$ that belongs to $D_k$ is such that $\{u,w'\}$ disconnects $u'$ from the rest of the graph.  That is, if one or both of $v_{i-1},v_{i+1}$ are in $U$, by the previous arguments we find the desired path in $W$. 
In any of the above cases, we find that $H\setminus \{u\}$ is connected, that is $T \notin \CC(\overbar{B})$. Contradiction.
\end{proof}

\begin{corollary}\label{rem:Tminusu}
Let $\overbar{B}$ be a graph satisfying Setup \ref{setup: chain}, and let $T\in \CC(\overbar{B})$. Then for any $u \in U\cap T$ we have $T'=T\setminus \{u\} \in \CC(\overbar{B})$. In particular, $\CC(\overbar{B})$ is an accessible set system.
 \end{corollary}
 \begin{proof}
 Let $a \in T'$. If $a \in U$, then from Lemma \ref{lem:compU} there exists $b \in W\cap T$ such that $\{a,b\} \in \CC(\overbar{B})$. In particular, $b \in T'$ and $c_{\overbar{B}}(T')> c_{\overbar{B}}(T'\setminus \{a\})$. If $a\in W$, namely $a$ is a cutpoint of $\overbar{B}$, then $c_{\overbar{B}}(T')> c_{\overbar{B}}(T'\setminus \{a\})$.\\
 Furthermore, for any non-empty $T \in \CC(\overbar{B})$ if $u \in T\cap U\neq \varnothing$, then $T'=T\setminus \{u\} \in \CC(\overbar{B})$, while if $T\cap U = \varnothing$, then any $w \in T$ is a cutpoint, hence $T\setminus \{w\}\in \CC(\overbar{B})$.
  \end{proof}

\begin{proposition}\label{prop:unmixedB}
Let $\overbar{B}$ be a graph satisfying Setup \ref{setup: chain}. Then $J_{\overbar{B}}$ is unmixed. 
\end{proposition}
\begin{proof}
We prove the statement by induction on $r$, the number of cycles in $\overbar{B}$. 

If $r=1$, then the claim follows. In fact, if $D_1=C_3$, then $\overbar{B}$ is a complete graph with or without whiskers, hence $J_{\overbar{B}}$ is unmixed by \cite[Proposition 2.6]{RR}. If $D_1=C_4$, then $\overbar{B}$ has to satisfy the condition (6) in Setup \ref{setup: chain}, and the resulting graph is known to be Cohen-Macaulay and hence unmixed. 

Suppose $r>1$. By induction hypothesis we have that $J_{\overbar{B}_k}$ is unmixed with  $B_k=\bigcup_{i=k}^r D_i$ and $k>1$.
 
If $D_1=C_3$ with $V(D_1)=\{u_0,u_1,w_1\}$ and $E(D_1)\cap E(D_2) =\{\{w_1,u_1\}\}$. Let $T \in \CC(\overbar{B})$. If $w_1 \notin T$, then $T$ is a cutset for $\overbar{B}_2$ and by induction hypothesis the assertion follows. We distinguish the following cases:
\begin{enumerate}
    \item  $w_1 \in T$ and $u_1 \notin T$;
    \item $w_1,u_1 \in T$.
\end{enumerate}

(1) Assume $w_1 \in T$ and $u_1 \notin T$. If $T$ is a cutset of $\overbar{B}_2$ the number of connected components does not change. In fact, by adding the graph $C_3$ and removing the vertex $w_1$ we only obtain that the connected component of $\overbar{B}_2\setminus T$ containing $u_1$ now contains the graph $D_1 \setminus w_1$. If $T \not \in \CC(\overbar{B}_2)$, we claim that $T' = T \setminus \{w_1\}$ is a cutset of $\overbar{B}_2$. We start observing that the connected component of $\overbar{B}_2 \setminus T$ containing $u_1$ contains $D_1\setminus\{w_1\}$ in $\overbar{B}\setminus T$. Since by hypothesis for any $a \in T'$  $c_{\overbar{B}}(T)>c_{\overbar{B}}(T\setminus \{a\})$ we have that $c_{\overbar{B}_2}(T')>c_{\overbar{B}_2}(T'\setminus \{a\})$, the claim follows.
Hence, by induction hypothesis,  $c_{\overbar{B}_2}(T')=|T'|+1$. Let $H$ be the connected component of $\overbar{B}_2 \setminus T'$ containing $w_1$. By adding the vertex $w_1$ to $T'$, $w_1$ disconnects $H$ into two connected components: the one containing $u_1$ and the free vertex attached to $w_1$.  


(2) If $w_1,u_1 \in T$, then there exists $v \in V(\overbar{B}_2)$ adjacent to $u_1$ such that $u_1$ breaks the connected component $H$ of $\overbar{B}\setminus (T\setminus \{u_1\})$ containing $u_1$ in two, one containing $v$ and  one containing $u_0$. By Setup \ref{setup: chain} (7), the vertices adjacent to $u_1$ in $\overbar{B_2}$ are either $w_1$ and $u$ or $w_1,w,$ and $u$. In the former case, since $w_1\in T$, then $u\notin T$ and $v=u$, otherwise $u_1$ is a free vertex in $D_1\setminus w_1$. In the latter case, $u,w\in V(D_3)$, that is $\{u,w\}\in \CC(\overbar{B}_2)$. We observe that $\{u,w\}\not\subset T$, otherwise $u_1$ is a free vertex of $D_1$. The claim follows. 

Moreover, from Corollary \ref{rem:Tminusu}, $T'=T\setminus \{u_1\}$ is a cutset of $\overbar{B}$ such that  $w_1 \in T'$ and $u_1 \notin T'$. By applying similar arguments to the case (1) we get that $c_{\overbar{B}}(T')=|T'|+1$ and $T' \cap \{u_1\}$ breaks the component containing $u_1$ in two: the vertex $u_0$, and the component containing the vertex $v$.

If $D_1 = C_4$ with $V(D_1)=\{u_0,w_0,u_1,w_1\}$, then $E(D_1)\cap E(D_2) =\{\{w_1,u_1\}\}$. Let $T \in \CC(\overbar{B})$. Assume $w_0,w_1 \notin T$, then $T$ is a cutset for $\overbar{B}_2$ and by induction hypothesis the assertion follows. We now assume $u_0,u_1 \notin T$ and since $\{w_0,u_1\}$ is the unique cutset of $B$ with cardinality $2$  containing $w_0$, then the cases $w_0 \in T $ and  $w_1 \notin T$, $w_0 \notin T$ and $w_1 \in T$, $w_0, w_1 \in T$ are analogous to the cases $w_1 \notin T$ and $w_1 \in T$ of $D_1=C_3$. In fact, in all of the cases we obtain that $T\setminus\{w_0\}$ is a cutset of $\overbar{B}$, that is $c_{\overbar{B}}(T\setminus \{w_0\})=|T\setminus \{w_0\}|+1$ and the component containing $u_0$ and $f_0$ is eventually broken by $w_0$. We now assume $u_1 \in T$. Observe that from Setup \ref{setup: chain} (2) $D_2=C_3$ and the vertex $u\in U$ adjacent to $u_1$ in $\overbar{B}_2$ is such that $\{w_1,u\}\in E(\overbar{B})$, otherwise $u_0,w_1,w_2,u$ are all adjacent to $u_1$ contradicting Setup \ref{setup: chain} (7). That is either $w_0$ or $w_1 \in T$, $u \notin T$, and from Corollary \ref{rem:Tminusu} $T \setminus \{u_1\}$ is a cutset of $\overbar{B}$. From the above cases, we obtain $c_{\overbar{B}}(T \setminus \{u_1\})=|T \setminus \{u_1\}|+1$ and $u_1$ breaks the component containing  $u_0$ and $u_2$. If $u_0 \in T$, then, by Lemma \ref{lem:compU}, $w_1 \in T$ and $w_0,u_1 \notin T$, that is from Corollary \ref{rem:Tminusu} $T\setminus \{u_0\}$ is a cutset for $\overbar{B}$. By the previous cases we obtain $c_{\overbar{B}}(T \setminus \{u_0\})=|T \setminus \{u_0\}|+1$ and $u_0$ breaks the component containing  $w_0$ and $u_1$.

\end{proof} 

\begin{remark}\label{rem:D1Kn}
In Proposition \ref{prop:unmixedB}, if we substitute $D_1$ with a complete graph $K_n$, with $n \geq 3$, satisfying (3) in the Setup \ref{setup: chain}, then $J_{\overbar{B}}$ is unmixed. 
\end{remark}

\begin{theorem}\label{theo:chaincycle}
Let $\overbar{B}$ be a graph. The following conditions are equivalent:
\begin{enumerate}
\item $\overbar{B}$ satisfies Setup \ref{setup: chain};
\item $J_{\overbar{B}}$ is Cohen-Macaulay;
\item $S/J_{\overbar{B}}$ is $S_2$;
\item $\overbar{B}$ is accessible;
\item $J_{\overbar{B}}$ is strongly unmixed.
\end{enumerate}

\end{theorem}
\begin{proof}
We prove the following implications:
\[
(5) \implies (2) \implies (3) \implies (4) \implies (1) \implies (5).
\]
\noindent By  \cite[Section 5]{BMS2}, it holds (5) $\implies$ (2). 

\noindent It is a well known result that (2) $\implies$ (3). 

\noindent Theorem \ref{theo:S2} states (3) $\implies$ (4).

\noindent By Lemmas \ref{lem:C3C4}, \ref{lem:C4nextC3}, \ref{lem:cutchain}, \ref{prop:center}, and observing that a $C_4$ with $2$ whiskers satisfying Setup \ref{setup: chain} (e) (or Setup \ref{setup: chain} (f)) is accessible, we have (4) $\implies$ (1).

\noindent To prove (1) $\implies$ (5) we proceed by induction on the number $s$ of cutpoints of $\overbar{B}$.

Let $s=1$ and $w$ be the cutpoint of $\overbar{B}$. Then $\overbar{B}$ is a cone from $w$ to exactly $2$ graphs: an isolated vertex and a path. By \cite{RR}, $J_{\overbar{B}}$ is unmixed. Moreover $\overbar{B}\setminus\{w\}$ is decomposable into edges, therefore $J_{\overbar{B}}$ is strongly unmixed by Lemma \ref{lem:strong_unmixed_decomposable},  and $\overbar{B}_{w}$ and $\overbar{B}_{w}\setminus \{w\}$ are complete graphs. 


Suppose $s>1$ and we focus on the cycle $D_1$. Let $w$ be the first cutpoint, namely $w=w_0$ if $D_1=C_4$ or $w=w_1$ if $D_1=C_3$. We observe that $\overbar{B}\setminus w=\pi\cup \overbar{B}_{t+1}$, where $\pi:u_0, u_1,\ldots, u_t$ is a path, $\{u_t\}=V(\pi)\cap V(\overbar{B}_{t+1})$, and $B_{t+1} = \bigcup_{i=t+1}^r D_i$. If $D_{t+1}=C_3$, then $\pi\cup \overbar{B}_{t+1}$ is decomposable in $u_t$. Note that $D_{t+1}$ cannot be a $C_4$. In fact, if by contradiction $D_{t+1}=C_4$, then $D_t=C_3$ and $u_{t-1},u_{t+1}, w, w_t$ are all adjacent to $u_t$. That is $\deg u_t\geq 4$ obtaining a contradiction and the claim follows. Therefore, by Lemma \ref{lem:strong_unmixed_decomposable} and by induction hypothesis, $J_{\overbar{B}\setminus w}$ is strongly unmixed. 

Now we prove that  $J_{\overbar{B}_{w}}$ is strongly unmixed, as well. Suppose $D_t=C_3$ then $\overbar{B}_{w}=K_{t+3}\cup \overbar{B}_{t+1}$ with $V(K_{t+3})\cap V(D_{t+1})=\{w_t,u_t\}$ and the associated binomial edge ideal is strongly unmixed by induction hypothesis. If $D_t=C_4$ with $V(D_t)=\{w_{t-1},w_t,u_{t-1},u_t\}$, then $\overbar{B}_{w}=K_{t+3}\cup D'_t\cup \overbar{B}_{t+1}$ where $D'_t=C_3$, $V(K_{t+3})\cap V(D'_t)=\{u_{t-1},w_t\}$ and $V(\overbar{B}_{t+1})\cap V(D'_t)=\{w_{t}, u_t\}$. We observe that 
$\overbar{B}_{w}$ satisfies Setup \ref{setup: chain} and Remark \ref{rem:D1Kn}. By induction hypothesis, the associated binomial edge ideal is strongly unmixed. It is straightforward to observe that 
$J_{\overbar{B}_{w}\setminus\{w\}}$ is strongly unmixed, too.



\end{proof}

\section{Computation of graphs with $n\in\{2,\ldots,12\}$ vertices}\label{sec: computation}

\begin{theorem}\label{Theo:ByComputer}
 Let $G$ be a graph on $[n]$, with $n\leq 12$. The following conditions are equivalent:
 \begin{enumerate}
  \item $S/J_G$ is Cohen-Macaulay;
  \item $S/J_G$ is $S_2$;
  \item $G$ is accessible;
  \item $J_G$ is strongly unmixed.
 \end{enumerate}

\end{theorem}
\begin{proof}
We know that 
\[
    (4) \implies (1) \implies (2) \implies (3)
\]
so, to prove the equivalence it is sufficient to show that $(3) \implies (4)$.

To prove the claim we have implemented a computer program that, for a fixed number $n$ of vertices, performs the following steps (steps \textit{(S2)}, \textit{(S3)} and \textit{(S4)} work on the result of the previous step):
\begin{itemize}[label={}]
\item \textit{(S1)} compute all connected non isomorphic graphs on $[n]$;   
\item \textit{(S2)} thanks to Lemma \ref{lem:strong_unmixed_decomposable}, keep only the graphs which are indecomposable and unmixed;
\item \textit{(S3)} keep only the ones that are accessible;
\item \textit{(S4)} keep only the ones that are strongly unmixed;
\item \textit{(S5)} verify that the graphs obtained from step \textit{(S3)} and \textit{(S4)} are the same.
\end{itemize}

The previous procedure was executed for the graph whose number of vertices is between $2$ and $12$. Finally, we refer readers to \cite{LMRR} for a complete description of the algorithm that we used.

\end{proof}

We underline that the computation of the graphs with $n=12$ vertices has been obtained in a month of computation on a node with 4 CPU Xeon-Gold 5118 having in total 48 cores and 96 threads.  All the graphs satisfying the equivalent conditions of Theorem \ref{Theo:ByComputer} are downloadable from \cite{LMRR}. Within this set we would like to focus on the graphs shown in the following example.
\begin{example}\label{example:wheel}
By direct computation we obtain the two graphs in Figure \ref{fig:W4W5}.

\begin{figure}[H]
\begin{subfigure}[t]{0.5\textwidth}
\centering
 \resizebox{!}{0.45\textwidth}{
\begin{tikzpicture}
\filldraw (0,0) circle (\rad);
\filldraw ( 0.707 , 0.707 ) circle (\rad);
\filldraw ( -0.707 , 0.707 ) circle (\rad);
\filldraw ( -0.707 , -0.707 ) circle (\rad);
\filldraw ( 0.707 , -0.707 ) circle (\rad);

\draw (0,0)--( 0.707 , 0.707 ) ;
\draw (0,0)--(-0.707 , 0.707 ) ;
\draw (0,0)--(0.707 , -0.707 ) ;
\draw (0,0)--(-0.707 ,-0.707 ) ;

\draw ( -0.707 , 0.707 ) --( 0.707 , 0.707 ) ;
\draw ( 0.707 , 0.707 ) --( 0.707 , -0.707 ) ;
\draw (-0.707 , -0.707 ) --(-0.707 , 0.707 ) ;
\draw ( 0.707 , -0.707 ) --( -0.707 , -0.707 ) ;

\filldraw (2*0.707 ,2* 0.707 ) circle (\rad);
\filldraw ( -2*0.707 ,2* 0.707 ) circle (\rad);
\filldraw ( -2*0.707 , -2*0.707 ) circle (\rad);
\filldraw ( 2*0.707 , -2*0.707 ) circle (\rad);

\draw ( -0.707 , 0.707 ) --(-2*0.707 , 2*0.707 ) ;
\draw ( 0.707 , 0.707 ) --( 2*0.707 , 2*0.707 ) ;
\draw (-0.707 , -0.707 ) --(2*-0.707 , 2*-0.707 ) ;
\draw ( 0.707 , -0.707 ) --( 2*0.707 , 2*-0.707 ) ;
\end{tikzpicture}}
\caption{}\label{fig:W4}
\end{subfigure}%
\begin{subfigure}[t]{0.5\textwidth}
\centering
 \resizebox{!}{0.5\textwidth}{
\begin{tikzpicture}
\filldraw (0,0) circle (\rad);
\filldraw (0, 1) circle (\rad);
\filldraw ( -0.951 , 0.309 ) circle (\rad);
\filldraw ( -0.588 , -0.809 ) circle (\rad);
\filldraw ( 0.588 , -0.809 ) circle (\rad);
\filldraw ( 0.951 , 0.309 ) circle (\rad);

\draw (0,0)--(0, 1);
\draw (0,0)--( -0.951 , 0.309 );
\draw (0,0)--( -0.588 , -0.809 );
\draw (0,0)--( 0.588 , -0.809 );
\draw (0,0)--( 0.951 , 0.309 );

\draw(0, 1)--( -0.951 , 0.309 ) ;
\draw ( -0.951 , 0.309 ) --( -0.588 , -0.809 );
\draw ( -0.588 , -0.809 )--( 0.588 , -0.809 );
\draw ( 0.588 , -0.809 )--( 0.951 , 0.309 );
\draw ( 0.951 , 0.309 )--(0,1);

\filldraw (0, 2) circle (\rad);
\filldraw ( 2*-0.951 , 2*0.309 ) circle (\rad);
\filldraw ( 2*-0.588 , 2*-0.809 ) circle (\rad);
\filldraw ( 2*0.588 , 2*-0.809 ) circle (\rad);
\filldraw ( 2*0.951 , 2*0.309 ) circle (\rad);

\draw(0, 1)--( 0,2 ) ;
\draw ( -0.951 , 0.309 ) --( 2*-0.951 , 2*0.309 );
\draw ( -0.588 , -0.809 )--(  2*-0.588 , 2*-0.809  );
\draw ( 0.588 , -0.809 )--( 2*0.588 , 2*-0.809 );
\draw ( 0.951 , 0.309 )--(2*0.951 , 2*0.309);

\end{tikzpicture}}
\caption{}\label{fig:W5}
\end{subfigure}
\caption{The accessible $\overbar{W}_n$.}\label{fig:W4W5}
\end{figure}
The graphs in Figure \ref{fig:W4W5} (A) and (B) are well known. In fact, the blocks that are not edges are the so-called \textit{wheel graphs} and they are denoted by $W_4$ and $W_5$, respectively. Whereas the blocks with whiskers are called \emph{Helm graphs} (see \cite{Helm}).

We observe that if $i>5 $ then $J_{\overbar{W}_i}$ is not unmixed. In fact, in this case we have at least $6$ vertices of degree $4$, say $v_1,\dots,v_6$. Without loss of generality,  we may assume that $\{v_i,v_{i+1}\}\in E(\overbar{W}_i)$, for $i=1,\dots,5$. Moreover, assume that $v$ is the vertex of degree $i$.
We can see that $T=\{v,v_1,v_3,v_5\}$ is a cutset such that $c(T)=6$.

\end{example}

We recall the following definition.

\begin{definition}\label{def:polyhedral}
A \textit{polyhedral graph} is a $3$-connected planar graph.
\end{definition}

The name of polyhedral derives from the fact that it is the graph whose vertices and edges are the ones of a convex polyhedron.

By Example \ref{example:wheel} and Definition \ref{def:polyhedral}, it is natural to ask
\begin{question}\label{question}
Is it possible to find an infinite family of accessible graphs $\overbar{B}$ such that $B$ is a polyhedral graph?
\end{question}

\begin{acknowledgement}
The computation of this work has been obtained thanks to the server of the Laboratory of Cryptography of the  Department of Mathematics, University of Trento.
\end{acknowledgement}

\end{document}